\documentclass[a4, 11pt]{amsart}
\usepackage{amssymb,latexsym}
\usepackage{color}
\usepackage[left=35mm, right=35mm, top=30mm, bottom=30mm, includefoot, includehead]{geometry}
\theoremstyle{plain}
\newtheorem{theorem}{Theorem}[section]
\newtheorem{corollary}[theorem]{Corollary}
\newtheorem{proposition}[theorem]{Proposition}
\newtheorem{lemma}[theorem]{Lemma}
\theoremstyle{definition}
\newtheorem{definition}[theorem]{Definition}
\newtheorem{example}[theorem]{Example}

\newtheorem{remark}[theorem]{Remark}

\DeclareMathOperator{\Res}{Res}

\DeclareMathOperator{\red}{red}

\DeclareMathOperator{\reg}{reg}

\newcommand{\enm}[1]{\ensuremath{#1}}          %

\newcommand{\cal}[1]{\mathcal{#1}}

\newcommand{\CC}{\enm{\mathbb{C}}}

\newcommand{\NN}{\enm{\mathbb{N}}}
\newcommand{\RR}{\enm{\mathbb{R}}}
\newcommand{\QQ}{\enm{\mathbb{Q}}}
\newcommand{\ZZ}{\enm{\mathbb{Z}}}

\renewcommand{\AA}{\enm{\mathbb{A}}}
\newcommand{\PP}{\enm{\mathbb{P}}}

\newcommand{\Aa}{\enm{\cal{A}}}
\newcommand{\Bb}{\enm{\cal{B}}}

\newcommand{\Ii}{\enm{\cal{I}}}

\newcommand{\Oo}{\enm{\cal{O}}}

\newcommand{\Rr}{\enm{\cal{R}}}
\newcommand{\Ss}{\enm{\cal{S}}}

\newcommand{\Ree}{\mathrm{Re}}
\newcommand{\Imm}{\mathrm{Im}}

\renewcommand{\phi}{\varphi}
\renewcommand{\theta}{\vartheta}
\renewcommand{\epsilon}{\varepsilon}

\renewcommand{\to}[1][]{\xrightarrow{\ #1\ }}

\newcommand{\old}[1]{}

\date{}
\title{The strength for line bundles}

\author{Edoardo Ballico and Emanuele Ventura}

\address{Universit\`a di Trento, 38123 Povo (TN), Italy}
\email{edoardo.ballico@unitn.it}
\address{Universit\"{a}t Bern, Mathematisches Institut, Sidlerstrasse 5, 3012 Bern, Switzerland}
\email{emanueleventura.sw@gmail.com, emanuele.ventura@math.unibe.ch}

\keywords{Strength; Line bundles; Complete intersections.}
\subjclass[2010]{14C20, 14N05, 14M10.}

\begin{document}

\maketitle

\begin{abstract}
We introduce the strength for sections of a line bundle on an algebraic variety. This generalizes the strength of homogeneous polynomials that has been recently introduced to resolve Stillman's conjecture, an important problem in commutative algebra. We establish the first properties of this notion and give some tool to obtain upper bounds on the strength in this framework. Moreover, we show some results on the usual strength such as the reducibility of the set of strength two homogeneous polynomials. 
\end{abstract}

\section{Introduction}

Recently, the notion of {\it strength} of polynomials has been introduced by Ananyan and Hochster in their groundbreaking work \cite{AH}, in order to resolve a longstanding conjecture in commutative algebra: Stillman's conjecture. 

In \cite{ESS}, Erman, Sam, and Snowden extended this definition to the {\it collective strength} of a set of polynomials $\lbrace f_i\rbrace$, as a way of measuring how much these $f_i$ behave alike independent variables. The principle that infinite collective strength is equivalent to the collection being a regular sequence has been successfully exploited by these authors, resulting into another proof of Stillman's conjecture, where they employ a very interesting ultraproduct construction.

In \cite{BDE}, Bik, Draisma, and Eggermont showed that strength of tensors has a remarkable property: 
for any proper Zariski closed subset $X(V)$, that is functorial in the underlying vector space $V$, the strength of every element in $X$ is uniformly bounded (independently of the underlying vector space $V$). This generalizes a result of Kazhdan and Ziegler \cite{KZ}, where, using techniques from additive combinatorics, they show that polynomials with directional derivatives of bounded strength have bounded strength themselves. More recently, Kazhdan and Ziegler studied properties of vanishing loci of high strength polynomials over finite fields \cite{KZ20}. Strength has also been a tool to derive the equivariant topological noetherianity result in \cite{DES}. 

In the setting of tensors, a concept similar to strength is the {\it slice rank} \cite{SawinTao}, that has been successfully employed in additive combinatorics: Ellenberg and Gijswijt \cite{EG} showed that a capset $A\subset \mathbb F_3^n$, i.e. a set not containing any arithmetic progression of length tree, has size $|A|\leq O(2.756^n)$. In fact, they obtain a more general result for such sets in any vector space $\mathbb F_q^n$. Their proof relies on an ingenious use of the slice rank of a specific tensor, encoding the size of any capset.

Inspired by all these deep results and applications related to strength, we introduce a strength in the context of line bundles on algebraic varieties, mimicking the original definition for homogeneous polynomials and tensors. We believe this could have as many interesting applications as strength of polynomials has been proven to possess. \\

\noindent {\bf Contributions and structure of the paper}. We now discuss more in detail the content of this paper. In \S\ref{strengthlinebundles}, given a pair $(X,L)$ of a complex projective variety $X$ and a line bundle $L$ on it, we introduce the strength for sections of $L$ (Definition \ref{defstrength}); this generalizes the strength of homogeneous polynomials and tensors, as explained in Examples \ref{strengthpoly} and \ref{strengthtensors} respectively. 

In Example \ref{weak}, we discuss a weaker version of strength for line bundles, which is then sensible to the difference between Weil and Cartier divisors. 
In Examples \ref{a4}, \ref{c3}, and \ref{d25}, we highlight several situations where we determine the strength of (sections corresponding to) Cartier divisors. 
In Proposition \ref{c1}, we somewhat generalize the case when $X$ is a rational normal curve, corresponding to the strength of binary forms, pointed out in Remark \ref{a1}. Theorem \ref{x3} is a tool to give upper bounds on the strength of a section, expressed in terms of vanishings of cohomology groups. The main ingredient behind it is the use of the Koszul complex. 

In \S\ref{mstrengthsec}, we introduce another central concept to our work, inspired by Lemma \ref{a7}. This is the {\it restricted strength} or $R$-strength (Definition \ref{mstrength}). In the restricted setting, we take into account only one line bundle in the original definition of strength; see Remark \ref{specialcaseofstrength}. 
The $R$-strength is a generalization of slice rank for homogeneous polynomials, as observed in Remark \ref{genslicerank}. 
Theorem \ref{e1} and Example \ref{111} highlight the first properties of this version of strength. In Remark \ref{d5}, we point out the connection between Segre varieties and loci of sections of $R$-strength one.  

In \S\ref{reals}, we discuss the validity of the some of the results proven in the previous sections over the field $\RR$. Theorem \ref{r2} shows a sufficient condition to be able to give a bound on the real strength in terms of the complex strength of a section. 

In \S\ref{specialvarieties}, we focus on specific pairs $(X,L)$. We explain how complete intersections have a role in the context of the strength of a homogeneous polynomial: in Remark \ref{cc2} we deal with complete intersections in $\PP^2$, leading to Theorem \ref{setdecomp}. Here we determine
the dimension of the set of strength decompositions of a given irreducible $u\in H^0(\Oo_{\PP^2}(d))$, that can be represented as $u = u_i f_{d-i} + v_j g_{d-j}$ where $\lbrace u_i, v_j\rbrace$ is a regular sequence of degrees $i$ and $j$. 

In \S\ref{strengthcompleteintersections}, we continue discussing more closely the case of strength for homogeneous polynomials and some relation to complete intersections. We define the locus $E(i_1,\ldots, i_k;n)$ in Definition \ref{defE}, and $F(i_1,\ldots, i_k; n)$ in Definition \ref{defF}. The first is the set of all $Y\in |\Oo_{\PP^n}(d)|$ (degree $d$ hypersurfaces) whose corresponding section (up to scaling) $f\in H^0(\Oo_{\PP^n}(d))$ admits a strength decomposition (not necessarily minimal) of the form 
\[
f =\sum _{h=1}^{k} f_hg_h, \mbox{ where } \deg(f_h) =i_h\leq \lfloor d/2\rfloor.
\]
The second locus is described by requiring the collection $\lbrace f_h\rbrace_h$ to be a regular sequence, i.e. their vanishing set $\lbrace f_1 = \cdots = f_h = 0\rbrace$ is a complete intersection. 

The loci  $F(i_1,\ldots, i_k; n)$ are irreducible (Remark \ref{ooo1}), albeit not generally closed. The latter statement is shown by constructing suitable flat limits of 
complete intersections of two quadrics in $\PP^3$: there exists a sequence of surfaces in $F(2,2;3)$, whose limit does {\it not} belong to $F(2,2;3)$ itself. 
This is achieved in Theorem \ref{cor1} and Theorem \ref{cor2}. To the best of our knowledge, this is a new and interesting result on its own right. 

Let $\mathcal{S}_k(d,n)$ be the set of all $Y\in |\Oo_{\PP^n}(d)|$ whose corresponding section $f\in H^0(\Oo_{\PP^n}(d))$ has strength $k$. Theorem \ref{d31} proves that $\mathcal{S}_2(d,n)$ is reducible. To establish this result, we make a detour to describe the {\it effective cone} of divisors of certain irreducible smooth surfaces (Proposition \ref{u100}).  
 
Finally, Example \ref{newd1} shows the existence of $u\in H^0(\Oo_{\PP^3}(d))$ such that $\mathrm{s}(u) < \mathrm{s}^R(u)$ for all line bundles $R$ on projective space, using a known description of the Noether-Lefschetz locus in $|\Oo_{\PP^3}(d)|$, briefly recalled in Remark \ref{newd1}. 
\vspace{3mm}

\begin{small}
\noindent {\bf Acknowledgements.} The first author is partially supported by MIUR and GNSAGA of INdAM (Italy). The second author is supported 
by the grant NWO Den Haag no. 38-573 of Jan Draisma. We thank A. Bik and A. Oneto for useful discussions and for kindly sharing with us their upcoming work \cite{BO}. 
\end{small}

\section{Strength of sections of line bundles}\label{strengthlinebundles}

Let $(X,L)$ be a pair consisting of a complex integral projective variety $X\subset \PP^N$ and a line bundle (i.e. a rank-one locally free sheaf) $L$ on $X$ with $H^0(L)> 0$, i.e. the vector space of global sections of $L$ on $X$ is nonzero. The set of nonzero sections is denoted $H^0(L)^{\times}$ and $|L|$ denotes
the linear system $\PP(H^0(L)^{\vee})$. Here is our central definition: 

\begin{definition}\label{defstrength}
Let $(X,L)$ be a pair as above, and let $u\in H^0(L)^{\times}$ be a section. The \emph{strength} $\mathrm{s}_{X}(u)$ of $u$
is the minimal number $r\geq 1$ such that there exist $2r$ {\it non-trivial} line bundles $L_i, R_i$ (with $1\le i\le r$), isomorphisms $\phi _i: L_i\otimes R_i \to L$ of line bundles, and sections $x_i\in H^0(L_i)$, $y_i\in H^0(R_i)$ such that 
\begin{equation}\label{strengthdec}
u = \sum _{i=1}^{r} \phi^0 _i(\mu_{L_i, R_i}(x_i\otimes y_i)), 
\end{equation}
where
\[
\mu_{L_i, R_i}: H^0(L_i)\otimes H^0(R_i) \rightarrow H^0(L_i\otimes R_i)
\]
is the multiplication of sections and $\phi^0_i$ denotes the isomorphism on global sections $\phi^0_i: H^0(L_i\otimes R_i)\cong H^0(L)$ induced by $\phi_i$. 
We say that  $\mathrm{s}_{X}(u) = +\infty$ if there exist no such $L_i, R_i,\phi _i$, and sections $x_i, y_i$ for any positive integer $r$. 
(We may drop the subscript from $\mathrm{s}_X$, once $X$ is specified.) 
\end{definition}

\begin{example}\label{strengthpoly}

The definition of strength for homogeneous polynomials is a special case of Definition \ref{defstrength}. 
Let $X = \PP(V)$ and $L = \Oo_{\PP(V)}(d)$. For the pair $(\PP^n, \Oo_{\PP(V)}(d))$, a nonzero section $f\in H^0(L)\cong S^d V$ is then identified with a nonzero homogeneous polynomial.  The strength of $f$ is defined to be the minimal $r$ such that
\[
f = \sum_{i=1}^r t_i s_i,
\]
where $t_i\in S^{e_i} V\cong H^0(\Oo_{\PP(V)}(e_i)), s_i \in S^{d-e_i} V\cong H^0(\Oo_{\PP(V)}(d-e_i))$. Therefore, in this case, the morphisms $\phi_i$ are the natural isomorphisms 
\[
\phi_i: \Oo_{\PP(V)}(e_i)\otimes \Oo_{\PP(V)}(d-e_i)\cong \Oo_{\PP(V)}(d) = L. 
\]
\end{example}

\begin{remark}\label{a1}
Over an algebraically closed field, every binary form is a product of linear forms. Thus
$\mathrm{s}_X(u)=1$ for all integers $d\ge 2$ and all $u\in H^0(\Oo _{\PP^1}(d))^{\times}$. For $d=1$, one has infinite strength,
as we prescribed that no $R_i$ or $L_i$ can be trivial. Similarly, for all $n\ge 2$, any nonzero element of $H^0(\Oo _{\PP^n}(1))$, i.e. any nonzero linear form, has infinite strength.\end{remark}

\begin{example}\label{strengthtensors}
A $d$-way tensor may be regarded as a global section of the line bundle $L = \Oo_{\PP^{n_1}}(1)\boxtimes \cdots \boxtimes \Oo_{\PP^{n_d}}(1)$ on the Segre
variety $X = \PP^{n_1}\times \cdots \times \PP^{n_d}$. 

The strength of a tensor $t$ (in the sense of \cite[1.5]{BDE}) coincides with the strength of $t$ viewed as a section of $L$: every possible choice for a line bundle $R_i$ in 
Definition \ref{defstrength} has the form 
$R_i \cong \boxtimes_{j\in J_i} \Oo_{\PP^{n_{j}}}(1) \boxtimes_{k\in [d]\setminus J_i} \Oo_{\PP^{n_{k}}}$, for some $J_i\subset [d]=\lbrace 1,\ldots, d\rbrace$. The {\it slice rank} of $t$ \cite{SawinTao} is lower bounded by its strength. 
\end{example}

Now, we briefly (and informally) recall Weil and Cartier divisors, referring to \cite[Chapter II.6]{h} or \cite[Chapter 1]{laz1} for details and precise results. On an integral scheme $X$ (e.g. an irreducible and reduced algebraic variety), a {\it Weil divisor} is a formal sum of codimension-one integral subschemes of $X$, called prime divisors. 

A {Cartier divisor} is a Weil divisor that can be described locally (in an open cover of $X$) by suitable invertible sections $f_i\in K(X)$. For a locally factorial scheme, Weil and Cartier divisors are equivalent \cite[Proposition II.6.11]{h}; for instance, this holds when $X$ is smooth. In this case, both Weil and Cartier divisors form the {\it class group} $\mathrm{Cl}(X)$ (modulo linear equivalence) of $X$, which is isomorphic to $\mathrm{Pic}(X)$, the group of invertible sheaves on $X$ \cite[Corollary II.6.14]{h}, i.e. line bundles or rank-one locally free sheaves. For non-locally factorial schemes, Cartier divisors form the {\it Cartier class group} $\mathrm{CaCl}(X)$ (modulo linear equivalence) of $X$, which is still isomorphic to $\mathrm{Pic}(X)$. However, Weil divisors are generally different from Cartier divisors even on normal varieties. \\

\noindent {\bf Notation}. In the rest, we use additive notation when we work with divisors $D\subset X$, and multiplicative notation for sheaves $\Oo_X(D)$.

\begin{example}\label{weak}
For a weaker version of strength for sections (called {\it weak strength} in this example), one may allow in Definition \ref{defstrength} the $R_i$ to be non-trivial rank-one {\it reflexive} sheaves. A reflexive sheaf is torsion-free but not necessarily locally free (so it is not necessarily a line bundle); see \cite{Schwede} for a nice introduction. 
Given a Weil divisor $D$, the sheaf $\Oo_X(D)$ is rank-one reflexive. 

Let $X$ be an integral normal projective variety. Let $D$ be an effective Weil divisor, $D\ne 0$. Thus the sheaf $\Oo_X(D)$ is a rank-one reflexive sheaf on $X$. 

A section corresponding to the divisor $D' = D + D = 2D$, which is a section of the reflexive sheaf $\Oo_X(2D)$ has weak strength one. 

However, $2D$ may still be a Cartier divisor, without $D$ being a Cartier divisor: the classical example is when $X\subset \PP^3$ is an irreducible quadric cone and  $D$ a line passing through the vertex \cite[Example 6.11.3]{h}. The divisor $2D$ is Cartier as $2D\in |\Oo _X(1)|$, but it is not the sum of two (or more) effective Cartier divisors. Let $u$ be a section of $\Oo_X(2D)$: the weak strength of $u$ is one, while $\mathrm{s}_X(u) =+\infty$, according to Definition \ref{defstrength}. 
\end{example}

\begin{example}\label{a2}
Let $(X,L)$ be a pair.  If $D\subset X$ is a sum of two Cartier divisors, then $\mathrm{s}_X(u) =1$ for any nonzero $u\in H^0(L)^{\times}$. Otherwise, $s(u) =+\infty$. 

Let $E\subset X$ be the base divisor of the linear system $|L|$, i.e. any divisor $D\in |L|$ contains $E$. 
If $E\ne 0$ and is Cartier, then $L_1 = L(-E)$ is a line bundle. Moreover, $R_1 = \Oo _X(E)$ is also a line bundle because $E$ is Cartier. Thus we have a natural isomorphism: 
\[
\phi_1: L(-E)\otimes \Oo_X(E)\rightarrow L. 
\]
So we see that $\mathrm{s}_X(u)=1$ for all $u\in H^0(L)^{\times}$. 
\end{example}

\begin{example}\label{a3}
Let $L$ be a line bundle such that its linear system $|L|$ has no base divisors. Fix a section $u\in H^0(L)^{\times}$ and let $D = \{u=0\}$ be the corresponding Cartier divisor of $X$. When $D = \sum _{i=1}^{s} a_i D_i$ with $a_i\geq 0$, $\sum_i a_i \ge 2$, and each $D_i$ is a Cartier divisor of $X$, then $\mathrm{s}_X(u)=1$, because we may use the line bundles $R_1:= \Oo _X(D_1)$ and $L_1:= \Oo _X(D-D_1)$, and once again with $\phi_1: L_1\otimes R_1\rightarrow L$ the natural isomorphism. 
\end{example}

A systematic way of constructing sections of infinite strength is as follows: 

\begin{remark}\label{d24}
Let $X$ be an integral projective variety of dimension $n > 0$. 

Fix any ample line bundle $H$ on $X$. For any effective Cartier divisor $D$ on $X$, the $H$-degree is the positive integer (intersection number) $D\cdot H^{n-1} >0$. This integer is invariant under linear (or numerical) equivalence. Let $D$ be any Cartier divisor of minimal $H$-degree, $L = \Oo _X(D)$ and let $u\in H^0(L)^{\times}$ be a corresponding section. 

The minimality of the $H$-degree of $D$ shows that $D \ne D_1+D_2$, where $D_i$ are non-zero Cartier divisors on $X$. Thus $\mathrm{s}_X(u) =+\infty$ for
all $u\in H^0(L)^{\times}$. 
\end{remark}

Another example which shows that in general there might be plenty of infinite strength sections: 

\begin{example}
Fix integers $n\ge 2$ and $t>0$. Let $B\subset \PP^n$ be a smooth hypersurface of even degree $2t$. Let $f:
X\to \PP^n$ be a double covering of $X$ branched over $B$. For any $a\in \ZZ$, define $\Oo _X(a) = f^\ast (\Oo _{\PP^n}(a))$.
Then $X$ is a smooth and connected projective variety and $\mathrm{Pic}(X) = \ZZ \Oo _X(1)$ \cite[7.1.20]{laz2}. Since $X$ is a double 
cover and $B$ has degree $2t$, one has $f_\ast (\Oo_X)\cong \Oo _{\PP^n}\oplus \Oo _{\PP^n}(-t)$ (this can be checked locally on affine schemes). 

By the projection formula \cite[Exercise II.5.1(d)]{h}, we have $f_\ast (f^\ast (\Oo _{\PP^n}(d)) \cong \Oo
_{\PP^n}(d)\otimes f_{\ast}(\Oo_X)\cong \Oo_{\PP^n}(d)\oplus \Oo _{\PP^n}(d-t)$. By definition of direct image, one has $H^0(\Oo _X(d)) = H^0(X, f^{\ast}(\Oo_{\PP^n}(d))) = H^0(\PP^n, f_{\ast}f^{\ast}\Oo_{\PP^n}(d))$. 
Therefore, $H^0(\Oo _X(d)) = f^{\ast} H^0(\Oo _{\PP^n}(d))$ for all $d<t$, whereas $h^0(\Oo
_X(t)) = \binom{n+t}{n}+1$. Thus a general section $u\in H^0(\Oo _X(t))$ has strength $\mathrm{s}_X(u) = +\infty$. 
\end{example}

The following situation shows that the picture may be richer for more interesting divisors on singular curves: 

\begin{example}\label{a4}
Let $X\subset \PP^2$ be an integral plane cubic with one singular point $o$, which is either an ordinary node or an ordinary cusp. In any case, its arithmetic genus is $p_a(X)=1$. The Riemann-Roch theorem for rank one torsion-free sheaves on $X$ shows that, for any scheme $S\subset X\setminus \{o\}$ with $\deg (S)=2$, the sheaf $L = \Oo _X(S)$ is a glabally generated line bundle (or equivalently, base-point free) and $h^0(L)=2$. Thus it comes with an associated degree two morphism $\psi: X\to \PP^1$. The degree two scheme $Z: = \psi^{-1}(\psi(o))\in |L|$ is Cartier.  We now show that its strength satisfies $\mathrm{s}_X(Z)=2$.

For any $q\in \PP^1$, the scheme $\psi^{-1}(q)$ is a degree two zero-dimensional scheme and $L\cong \Oo_X(\psi^{-1}(q))$.
For a general $p\in \PP^1$, the scheme $\psi^{-1}(p)$ is the union of two distinct points $z_1, z_2 \in X\setminus \{o\}$ and hence we have infinitely many divisors $D = z_1+ z_2\in |L|$. Since each $z_i$ is a smooth point of $X$, $D$ is an effective Cartier divisor of $X$, and therefore the section $u\in H^0(L)$ associated to $D$ is of the form $u=x_1\otimes x_2$ with $x_1\in H^0(\Oo_X(z_1))$ and $x_2\in H^0(\Oo_X(z_2))$. Thus $u$ has strength one. For the vector space $H^0(L)$ there is a basis formed by strength one sections. Since $h^0(L) =2$, we derive $\mathrm{s}_X(v)\le 2$, for all $v\in H^0(L)^{\times}$. 

Thus $\mathrm{s}_X(Z)\leq 2$ as well. Assume $\mathrm{s}_X(Z)=1$ and take line bundles $L_1$ and $R_1$, with sections $x\in H^0(L_1)$, $y\in H^0(R_1)$ such that $L_1\otimes R_1\cong L$ and $x\otimes y$ is a section whose zero-locus is $Z$. The sections $x$ and  $y$ are associated to degree one Cartier divisors $D_x$  and $D_y$, respectively. Hence both are associated to a smooth point of $X$, so $D_x+D_y= \psi^{-1}(o_1)$ for some $o_1\in \PP^1\setminus \{\psi(o)\}$, which is a contradiction.
\end{example}

\begin{example}\label{c3}
Let $X\subset \PP^2$ be a singular integral cubic and let $o$ be its singular point as in Example \ref{a4}. 
\begin{enumerate}

\item[(i)] Assume $X$ is cuspidal. Let $\ell\subset \PP^2$ be the unique line (i.e. the tangent) such that the scheme-theoretic intersection $D = \ell\cap X$ is
the divisor $3o$ of $X$; let $u$ be a corresponding section of $D$. 

Since $o$ is not a Cartier divisor of $X$, one has $\mathrm{s}_X(u)\ne 1$. Call $D_2$ the unique degree two subscheme of $D$ and $u_2$ a corresponding section.

\item[(ii)] Assume that $X$ is nodal. There are two lines $\ell_1, \ell_2$ (i.e. the tangents) such that the scheme-theoretic intersection
$D' = \ell_i\cap X$  is the divisor $3o$ of $X$; let $u'$ be a corresponding section. Since $o$ is not Cartier, we have $\mathrm{s}_X(u')\ne 1$. Call $D'_2 \subset D'$ the unique degree $2$ subscheme of $D'$, and similarly $u'_2$ a section. 
\end{enumerate}

In both (i) and (ii), $D_2$ and $D'_2$ cannot be written as a sum of Cartier divisors, as the only such are
smooth points. Therefore $\mathrm{s}_X(u_2) = \mathrm{s}_X(u'_2)= +\infty$. Moreover, $D_2$ and $D'_2$ are degree two Weil (but not Cartier) divisors that are {\it limits} of Cartier divisors of the form $p+q$, with $p, q \in X\setminus \{o\}$, of strength one. 
\end{example}

The next observation slightly generalizes Example \ref{c3}: 

\begin{remark}\label{c4}
Let $X\subset \PP^N$ be an integral and non-degenerate projective curve, and fix an integer $d\ge 2$. Let $A_{d}$ be the set of all degree $d$ effective Weil divisors which are a flat limit of a family of effective degree $d$ divisors contained in the smooth locus $X_{\reg}$ of $X$. Let $D\subset X_{\reg}$ be an effective degree $d$ divisor. Thus $D = p_1+\cdots +p_{d}$ with $p_i\in X_{\reg}$ (here we allow the case $p_i=p_j$ for some $i\ne j$). Since $p_i\in X_{\reg}$, each $\Oo _X(p_i)$ is a line bundle. Since $d\ge 2$, we have $\mathrm{s}_X(D) =1$. Every $D'\in A_{d}$ is a limit of Cartier divisors of strength one. It is generally difficult to describe the elements in $A_d$: it is the set of all smoothable degree $d$ zero-dimensional schemes supported on $X$. The set $A_2$ contains (and sometimes it is equal to) the limit of tangent vectors to smooth points of $X$, as for $D_2, D_2'$ in Example \ref{c3}.
\end{remark}

\begin{example}\label{d25}
Let $X$ be an integral projective variety with $\dim X \ge 2$. We show the existence of $L$ and $u$ such that 
$\mathrm{s}_X(u) =2$. We start by an observation: 

{\it Observation.} Let $D\ne 0$, be an effective Cartier divisor of $X$. Define $L = \Oo _X(D)$ and let $u\in
H^0(L)^{\times}$ be a section associated to $D\in |L|$. If $D$ is an integral scheme, then 
$\mathrm{s}_X(u) \ne 1$.

For a vector space $V\subset H^0(L)$, denote $|V|\subset |L|$ the corresponding linear system. Let $L$ be a line bundle such that $h^0(L)\ge 2$, and assume the existence of a two-dimensional linear subspace $V\subseteq H^0(L)$ such that a general $D\in |V|$ is integral, and that there are at least two distinct $D_1, D_2\in |V|$, each of them being the sum of effective Cartier divisors. Let $u$, $u_1$, and $u_2$ be sections corresponding to $D$, $D_1$, and $D_2$, respectively. By definition,  $\mathrm{s}_X(u_1) = \mathrm{s}_X(u_2) = 1$. Since $\dim V=2$, $u$ is a linear combination of $u_1$ and $u_2$. Thus $\mathrm{s}_X(u) \le 2$. By the {\it Observation}, one has $\mathrm{s}_X(u) \ne 1$, and so $\mathrm{s}_X(u) = 2$.

We show the existence of $L$, $V$, $D$, $D_1$, and $D_2$ as above on $X$. Fix any very ample line bundle $M$ on $X$ and let $L = M^{\otimes 2}$. A general $A\in |M|$ is integral by Bertini's theorem and hence $|L|$ is spanned by divisors $A_1+A_2$, with each $A_i$ integral. 

Let $W\subseteq H^0(L)$ be the image of the multiplication map $H^0(M)\otimes H^0(M)\to H^0(L)$. By Bertini's theorem, a general $G\in |W|$ is integral. Let $U \subseteq W$ be the set of all $D_i\in W$ with $D_i = A_i+B_i$, where $A_i,B_i\in |M|$ with $A_i, B_i$ integral divisors. Since a general element of $|M|$ is integral, the definition of $W$ shows that $U$ spans $W$. Each section in the linear span of two elements of $U$ has strength $\le 2$. If one of them has strength $2$, we are done. If not, then all such $u\in U$ have strength $1$. Therefore, we obtain that all $u\in W\setminus \{0\}$ in the linear span of three elements of $U$ have strength $\le 2$. If none of them has strength $2$, we keep continuing. After finitely many steps, we have to find some $u\in W$ with strength exactly $2$, as $\dim W$ is finite and we may apply the {\it Observation} above to a general divisor $D\in |W|$.
\end{example}

With the aim of generalizing Remark \ref{a1}, we introduce the following: 

\begin{definition}\label{c0=}
Let $Y$ be an integral projective variety and $L$ be a very ample line bundle on it. Let $X'\subset \PP^M$ be the image of $Y$ by the
complete linear system $|L|$. For any linear subspace $V\subset \PP^M$, let $\pi _V: \PP^M\setminus V \to \PP^N$, $N= M-\dim V
-1$, denote the linear projection away from $V$. Assume $V\cap X' =\emptyset$ and that the morphism $\pi _{V|X'}$ is an embedding, 
and let $X = \pi _V(X')$. 

Any $q\in \PP^M$ may be seen as a section of $L$, so the strength $\mathrm{s}_{X'}(q)$ is defined 
as the strength of that section. Moreover, for any $p\in \PP^N$, we may define the $X$-strength $\mathrm{s}_X(p)$ of $p$ 
as follows: 
\[
\mathrm{s}_X(p) = \min\lbrace \mathrm{s}_{X'}(q') \ | \ q'\in \PP^M\setminus V, \ \pi _V(q')= p \rbrace.
\]
\end{definition}

Equipped with this definition, the next result may be seen as a more general version of Remark \ref{a1}. 

\begin{proposition}\label{c1}
Let $X\subset \PP^N$, $N\ge 2$, be a smooth and irreducible non-degenerate curve. Then $\mathrm{s}_X(q) =1$ for all $q\in \PP^N$.
\end{proposition}

\begin{proof}
By Definition \ref{c0=}, it is sufficient to prove the statement when $X$ is linearly normal. Let $D\in |\Oo _X(1)|$ be the
effective divisor associated to $q$. Write $D = \sum a_ip_i$ for some $p_i\in X$, with integers $a_i> 0$
and $\sum_i a_i =\deg (X)$. Since $N\ge 2$ and $X$ is non-degenerate, we have $\deg (X)\ge 2$. Since $X$ is smooth, $\Oo _{X}(p_1)$ and $\Oo_X(1)(-p_1)$ are line bundles. Since $\sum_i a_i=\deg(X) \ge 2$, these two line bundles are non-trivial. The divisor $D$ is an element of $|\Oo _X(1)|$
associated to a rank-one element in the image of the multiplication map $H^0(\Oo_X(p_1))\otimes H^0(\Oo _X(1)(-p_1))$, which gives
the statement. 
\end{proof}

Likewise Remark \ref{a1} our version of it is sensitive to fields extensions: 

\begin{example}\label{fieldext}
Let $X\subset \PP^2$ be the smooth conic $\{x^2+y^2+z^2 =0\}$ seen as a smooth projective curve defined over $\RR$. Since $X(\RR)=\emptyset$, degree one line bundles are not defined over $\RR$. By construction $\Oo _X(1)\cong \Oo_{\PP^1}(2)$ has degree two and is defined over $\RR$. Since the degree one line bundles on $X$ are not defined over $\RR$, for all the sections $u\in H^0(\Oo _X(1)))^{\times}$, their strength defined over $\RR$ is $\mathrm{s}_{X(\RR)}(u) = +\infty$. Note that $X$ is geometrically irreducible, i.e. $X(\CC)$ is irreducible. By Proposition \ref{c1}, we have $\mathrm{s}_{X(\CC)}(u) =1$ for all $u\in H^0(X(\CC),\Oo _{X(\CC)}(1))^{\times}$. See \S\ref{reals} for more results and examples over the reals. 
\end{example}

Since strength for line bundles is defined through multiplication of sections,  the following lemma is very useful to derive some information about it, and heavily inspired the introduction of $R$-strength studied in \S\ref{mstrength}. The dual of a line bundle $R$ is denoted $R^{-1}$. 

\begin{lemma}\label{a7}
Let $(X,L)$ be a fixed pair. Let $R \ncong L$ be a non-trivial globally generated line bundle with $d = h^0(R)>0$ such that
multiplication map 
\[
H^0(R)\otimes  H^0(L\otimes R^{-1})\longrightarrow H^0(L)
\] 
is surjective. Then any non-zero element of $H^0(L)$ has strength at most $d$. 
\end{lemma}

\begin{example}\label{a8}
Let $X$ be a scheme such that $\mathrm{Pic}(X)\cong \ZZ H$ and take a positive power $M$ of $H$. Assume $M$ is very ample and projectively normal. Then, for any integer $d\ge 2$, any non-zero section of $H^0(M^{\otimes d})$ has strength $\leq h^0(M)$. 
This situation covers for instance Grassmannians, smooth complete intersections of dimension $\ge 3$, and very general surfaces of degree $\ge 4$.

See also \cite[\S 1.8]{laz1}, for related results in the setting of (regular) coherent sheaves on projective space. 
\end{example}

Let $X$ be an $n$-dimensional integral projective variety. We show next that, for all sufficiently positive line bundles $L$
on $X$, one has $\mathrm{s}_X(u) \le n+1$ for all $u\in H^0(L)^{\times}$. The example $(\PP^n,\Oo _{\PP^n}(1))$ shows that we need
to give some restriction on $L$. Recall that $M^{-k}$ denotes the dual line bundle of $M^{\otimes k}$. 

\begin{theorem}\label{x3}
Let $M$ be a globally generated non-trivial line bundle on $X$ and define $m = \min \{n+1,h^0(M)\}$. Let $L$ be a line bundle on $X$ such that the cohomology groups $H^k\left(M^{-k}\otimes L\right) = H^k\left(M^{-k-1}\otimes L\right) = 0$ for all $1\leq k\leq m-1$ and $H^m\left(M^{-m}\otimes L\right) = 0$. Then $\mathrm{s}_X(u)\le m$ for all $u\in H^0(L)^{\times}$. 
\end{theorem}

\begin{proof}
Let $W\subseteq H^0(M)$ be a general $m$-dimensional linear subspace. To show the statement is enough to check that there is a surjective map: 
\[
\mu_{W,L}: H^0(L\otimes M^{-1})\otimes W \longrightarrow H^0(L)\longrightarrow 0. 
\]

If $h^0(M)\le n+1$, then $m=h^0(M)$ and so $W$ spans $M$ by definition of globally generated sheaf. Suppose $h^0(M) \geq n+2$ and so $m=n+1 = \dim X +1$. 
Since $M$ is globally generated, the linear system $|M|$ induces a unique morphism $\phi_M: X\rightarrow \PP(H^0(M)^{\vee})$ \cite[Theorem II.7.1(b)]{h}. 
Since $h^0(M)>m$, we may apply a generic projection from $\PP(H^0(M)^{\vee})$ to $\PP(W^{\vee})$ that restricts to a morphism on $\phi_M(X)$, i.e. $\pi: X\rightarrow \PP(W^{\vee})$. Therefore $\pi\circ \phi_M$ is a morphism on $X$ and so this implies that $W$ spans the line bundle $M$, as for each point on $X$ we may find a section that does not vanish on it. 

Thus in either case, $W$ spans $M$. Upon fixing a basis of sections in $W$, we consider the sequence of line bundles on $X$:
\[
W\otimes \Oo_X \longrightarrow M \longrightarrow 0,
\] 
which is then exact, by the first paragraph of this proof. This gives a surjective morphism $W\otimes M^{-1}\rightarrow \Oo_X \rightarrow 0$. Now, consider the Koszul complex of the sections of $W$ we picked above: 
\[
\mathcal K: \cdots \longrightarrow \bigwedge^2 W\otimes M^{-2}\longrightarrow W\otimes M^{-1}\longrightarrow \Oo_X \longrightarrow 0. 
\]
Furthermore, consider the exact complex $\mathcal K_L = \mathcal K\otimes L$.  Thus: 
\[
\mathcal K_L: \ldots \stackrel{\phi_2}{\longrightarrow} \bigwedge^2 W\otimes M^{-2}\otimes L \stackrel{\phi_1}{\longrightarrow} W \otimes M^{-1}\otimes L\stackrel{\phi_{W,L}}{\longrightarrow L}\longrightarrow 0,
\]
where $\phi_k: \bigwedge^{k+1} W\otimes M^{-k-1}\otimes L \longrightarrow \bigwedge^{k} W\otimes M^{-k}\otimes L$ is the induced differential from the Koszul complex,
and $\phi_{W,L}$ is the multiplication of sections. (The map $\mu_{W,L}$ above is the map on global sections induced by $\phi_{W,L}$.)
Splitting the complex $\mathcal K_L$ into short exact sequences, at the first step one has: 
\[
0 \longrightarrow \mathrm{Ker}(\phi_{W,L}) \longrightarrow W\otimes M^{-1}\otimes L \longrightarrow L\longrightarrow 0.
\]
Taking global sections, this sequence implies that the map $\mu_{W,L}$ is surjective if and only if $H^1(\mathrm{Ker}(\phi_{W,L}))=0$. For $1\leq k\leq m$, we have the exact complexes: 
\[
0 \longrightarrow \mathrm{Ker}(\phi_{k}) \longrightarrow \bigwedge^{k+1}W\otimes M^{-k-1}\otimes L \longrightarrow \mathrm{Ker}(\phi_{k-1}) \longrightarrow 0,
\]
where we put $\phi_0 := \phi_{W,L}$. 
Taking the long exact sequences in cohomology of each of these exact sequences, for each $1\leq k\leq m$, we derive the exact complexes: 
\[
\bigwedge^{k+1} W \otimes H^k (M^{-k-1}\otimes L) \rightarrow H^k(\mathrm{Ker}(\phi_{k-1})) \rightarrow H^{k+1}(\mathrm{Ker}(\phi_{k})) \rightarrow \bigwedge^{k+1} W\otimes H^{k+1}(M^{-k-1}\otimes L).
\]
By assumption on the cohomological vanishings, the exact complex above is: 
\[
0 \rightarrow H^k(\mathrm{Ker}(\phi_{k-1})) \rightarrow H^{k+1}(\mathrm{Ker}(\phi_{k})) \rightarrow 0.
\]
Therefore $h^k(\mathrm{Ker}(\phi_{k-1})) = h^{k+1}(\mathrm{Ker}(\phi_{k}))$, for each $1\leq k\leq m$. Moreover, for $k=m$, the exact
complex is: 
\[
0 \rightarrow H^m(\mathrm{Ker}(\phi_{m-1})) \rightarrow H^{m+1}(\mathrm{Ker}(\phi_{m})) \rightarrow 0, 
\]
where $\mathrm{Ker}(\phi_{m})=0$, thus $H^{m+1}(\mathrm{Ker}(\phi_{m}))=0$ and $H^m(\mathrm{Ker}(\phi_{m-1}))=0$. This implies $H^1(\mathrm{Ker}(\phi_{W,L}))= 0$ and so the surjectivity of the map $W\otimes H^0(L\otimes M^{-1})\longrightarrow H^0(L)$ is proven.
\end{proof}

Let $M$ be a non-trivial line bundle on $(X,L)$ such that $M\ncong L$. For any linear subspace $W\subseteq H^0(M)$, let
\[
\mu_{W,L} := \mu _{{L\otimes M^{-1}, M}_{|W}}: H^0(L\otimes M^{-1})\otimes W \rightarrow H^0(L) 
\]
be the restriction to $W$ of the multiplication map of sections. As in the proof of Theorem \ref{x3}, when $W$ spans $M$, the image $\mathrm{Im}(\mu_{W,L})$ may be be studied by the means of the Koszul complex. The case $\dim W=2$ is related to the classical {\it base-point free trick}.

\begin{corollary}\label{d1}
Let $m = \dim W$. Then $\dim \mathrm{Im}(\mu_{W,L}) = \sum_{k=1}^{m} (-1)^{k-1} \binom{m}{k}h^0(M^{-k}\otimes L)$. 
\end{corollary}

\begin{proof}
This is a calculation using the complex $\mathcal K_L$ in the proof of Theorem \ref{x3}. 
\end{proof}

\begin{example}
Theorem \ref{x3} may be applied to many cases. In particular, for any $X$ and any globally generated line bundle $M$ on $X$, the vanishings in the assumption are satisfied for a sufficiently positive $L$ by a theorem of Serre. In fact, if $X$ is smooth, by Kodaira's vanishing, it would be sufficient to have the ampleness of $L\otimes \omega _X^{-1} \otimes M^{-k}$ for every $1\leq k\leq n$. 

For the sake of an example, suppose $\mathrm{Pic}(X) \cong \ZZ H$ with $H$ ample (e.g. $X$ satisfies this condition if it is a smooth complete intersection of dimension $\geq 3$ in projective space). Write $M \cong H^{\otimes m}$, $L \cong H^{\otimes \ell}$ and $\omega _X\cong H^{\otimes e}$ (in multiplicative notation). The ampleness of $L\otimes \omega _X^{-1} \otimes M^{-k}$ is guaranteed when $\ell -mk - e>0$.

\end{example}

\section{The $R$-strength}\label{mstrengthsec}

\begin{definition}\label{mstrength}
Let $(X,L)$ be a pair as above. Let $R$ be a non-trivial line bundle on $X$ such that $R\ncong L$. For any linear subspace $W\subseteq H^0(R)$, let
\[
\mu_{W,L} := \mu _{{L\otimes R^{-1}, R}_{|W}}: H^0(L\otimes R^{-1})\otimes W \rightarrow H^0(L) 
\]
be the restriction to $W$ of the multiplication map of sections. For any $u\in H^0(L)^{\times}$, the {\it restricted strength} or 
$R$-{\it strength} of $u$ is the minimal dimension $\dim W$ of some linear subspace $W\subseteq H^0(R)$ such that $u\in
\mathrm{Im}(\mu _{W,L})$ and is denoted $\mathrm{s}^{R}(u)$. Equivalently, $\mathrm{s}^{R}(u)$ is the minimal 
$r$ such that 
\[
u = \sum_{i=1}^r \phi^0\left(\mu_{L\otimes R^{-1}, R}(x_i\otimes y_i)\right),
\]
where $\phi^0$ is the isomorphism on global sections induced by $\phi: L\otimes R^{-1}\otimes R\rightarrow L$, $x_i\in H^0(L\otimes R^{-1})$, and 
$y_i\in H^0(R)$. We say that $\mathrm{s}^{R}(u)=+\infty$ if there is no such a linear subspace $W\subseteq H^0(R)$. 
\end{definition}

\begin{remark}\label{specialcaseofstrength}
Note that Definition \ref{mstrength} is a special case of the strength for the sections of $L$. Indeed, in Definition \ref{defstrength}, 
take $L_i = L\otimes R^{-1}, R_i = R$, and $r = \dim W$ whenever it is defined for a given $u$. Therefore $1\leq \mathrm{s}(u)\leq \mathrm{s}^R(u)$. 

If $u\in H^0(L)^{\times}$ has strength $\mathrm{s}(u) = 1$, then there exists a line bundle $R$ such that $\mathrm{s}^R(u) = 1$. Moreover, if $\mathrm{s}^R(u) = +\infty$ for all $u\in H^0(L)^{\times}$ and all line bundles $R$, then $\mathrm{s}(u) = +\infty$ for all $u\in H^0(L)^{\times}$. 

Example \ref{newd1} shows that in general $\mathrm{s}(u)< \mathrm{s}^R(u)$ for all $R$ line bundles on the given pair $(X,L)$. 
\end{remark}

\begin{remark}\label{genslicerank}
The definition of $R$-strength is a generalization of the notion of (symmetric) {\it slice rank} \cite{SawinTao} of homogeneous polynomials. 
For the first, consider the pair $X = \PP(V), L = \Oo_{\PP(V)}(d)$, and let $f\in H^0(\Oo_{\PP(V)}(d))\cong S^d V$. Instead of allowing 
line bundles $\Oo_{\PP(V)}(e_i)$ for arbitrary integers $1\leq e_i\leq d-1$ as in Example \ref{strengthpoly}, a slice decomposition of a homogeneous polynomial
$f$ requires to choose $e_i = 1$ for all $i$. This coincides with the restricted strength of $f$, taking $R = \Oo_{\PP(V)}(1)$ in Definition \ref{mstrength}. 
\end{remark}

\begin{theorem}\label{e1}
Assume that $\mathrm{Pic}(X)$ is finitely generated. Fix $L\in \mathrm{Pic}(X)$ such that $H^0(L) \ne 0$. 
Let $\Rr$ be the set of all $R\in \mathrm{Pic}(X)$ such that $h^0({R}) >0$,  $h^0(L\otimes R^{-1}) >0$, $R\ncong \Oo_X$ and $R\ncong L$. Then:

\begin{enumerate}
\item[(i)] If $\Rr =\emptyset$, then $s_X(u)=+\infty$ for all $u\in H^0(L)^{\times}$.

\item[(ii)] Assume $\Rr \ne \emptyset$. Then either there is a countable intersection $U$ of non-empty open subsets of $H^0(L)$ such that $s(u) = \mathrm{s}^R(u)=+\infty$ for all $u\in U$, or for each $R\in \Rr$ there is a non-empty open subset $U_{R}$ of $H^0(L)$ such that $\mathrm{s}^R(u)$ is constant (possibly infinite) for all $u\in U_{R}$.
\end{enumerate}
\end{theorem}

\begin{proof}
The $R$-strength is only defined if there exists $R\in \Rr$. Thus part (i) is clear.\\ 
\noindent For statement (ii), fix $R\in \Rr$. We have a regular map $\psi_R: \PP(H^0(R)) \times \PP(H^0(L\otimes R^{-1})) \rightarrow \PP(H^0(L))$, as explained in Remark \ref{d5}. Thus $\mathrm{Im}(\psi_R)$ is a closed projective subvariety in $\PP(H^0(L))$. 

Let $\Aa _R\subseteq H^0(L)$ be the affine cone over the projective span of $\mathrm{Im}(\psi_R)$, and define $\Bb _R =  H^0(L)\setminus \Aa _R$. The set $\Bb _R$ is an open set in $H^0(L)$ (possibly empty). By definition, $\Aa _R$ is the set of all $u\in H^0(L)^{\times}$ with finite $R$-strength, i.e. all points of $\Aa _R$ have $R$-strength $\le h^0({R})<+\infty$. Since for each positive integer $k$ the Grassmannian of all $k$-dimensional linear subspaces of $H^0(R)$ is complete, an application of Chevalley's theorem \cite[Ex.II.3.18,II.13.19]{h} gives the existence of a non-empty Zariski open subset $U_{R}$ of $\Aa _R$ such that  $\mathrm{s}^R(u) = k_R$ for all $u\in U_{R}$. 

Let $\mathcal E$ be the set of $R\in \Rr$ such that $\Bb _R =\emptyset$. If $\mathcal E =\emptyset$ we take $U = \cap _{R\in \Rr} \Bb _R$.  In this case, we have $\mathrm{s}^R(u) =+\infty$ for all $u\in U$ and all $R\in \Rr$. Thus $\mathrm{s}(u) = +\infty$ for all $u\in U$. 

Assume $\mathcal E\ne \emptyset$. For each $R\in \mathcal E$, the set $U_{R}$ is a non-empty open subset of $H^0(L)$ and all $u\in U_{R}$ have constant (and finite) $R$-strength. For each $R\notin \mathcal E$, we have $\Bb_R\neq \emptyset$. Then we define $U_R = \Bb_R$, which consists of sections $u$ such that $\mathrm{s}^R(u) = +\infty$. 
\end{proof}

\begin{example}\label{111}
Let $X$ be a smooth curve of genus $g\ge 0$ and $L$ be a line bundle on $X$ such that $h^0(L) >0$ of degree $d = \deg (L)$. 
\begin{enumerate}

\item[(i)] Suppose $h^0(L) =1$. We have
$|L| =\{D\}$ for an effective divisor $D\ne 0$ on $X$, $D =m_1p_1+\cdots +m_rp_r$ with $m_i>0$, $p_i\ne p_j$ for all
$i\ne j$ and $m_1+\cdots +m_r =d$. If $d=1$, then $\mathrm{s}_X(u)=+\infty$ for all $u\in H^0(L)^{\times}$. Let $\Rr _D$ be the set all isomorphism classes of line bundles 
$\Oo _X(a_1p_1+\cdots a_rp_r)$ with $0 \le a_i \le m_i$ for all $i$ and $0< a_1+\cdots +a_r <d$. For any $u\in H^0(L)^{\times}$, one has $\mathrm{s}^R(u)=1$ if $R\in \Rr _D$ and $\mathrm{s}^R(u)>1$ if $R\notin \Rr_D$. \\

\item[(ii)] Suppose $h^0(L)\ge 2$.  Let $\Rr$ be the set of all line bundles $R$ on $X$ such that  $h^0({R})>0$,
$h^0(L\otimes R^{-1})>0$, $R\ncong \Oo _X$ and $R\ncong L$. If $\Rr =\emptyset$, then $\mathrm{s}_X(u)=+\infty$ for all $u\in
H^0(L)^{\times}$ by Theorem \ref{e1}(i). (If $g=0$, then $\Rr \ne \emptyset$ if and only if $d\ge 2$ and $\Rr =\{\Oo
_{\PP^1}(t)\}_{0<t<d}$. For any integer $t$ such that $1\le t\le d-1$, each $u\in H^0(L)^{\times}$ has
$\Oo _{\PP^1}(t)$-strength $1$.)

By Proposition \ref{c1}, all $u\in H^0(L)^{\times}$ have $\mathrm{s}_X(u)=1$. We now consider the $R$-strength of $u$ for a fixed $R\in \Rr$. For any $u\in H^0(L)^{\times}$, set $D_u = \{u=0\}\in |L|$.  Fix any $D\in |L|$ and write $D =m_1p_1+\cdots +m_rp_r$ with $m_i>0$ and $p_i\ne p_j$, where $d =m_1+\cdots +m_r$. Again, let $\Rr _D$ be the set of all isomorphism classes of line bundles $\Oo _X(a_1p_1+\cdots +a_rp_r)$ with $0 \le a_i \le m_i$ for all $i$ and $0< a_1+\cdots +a_r <d$. Note that $\Rr = \bigcup _{D\in |L|} \Rr _D$. For any $u\in H^0(L) ^{\times}$ we have $\mathrm{s}^R(u) =1$ if $R\in \Rr _{D_u}$, whereas $\mathrm{s}^R(u) >1$ otherwise. \\

\item[(iii)] Assume $|L|$ has a base locus $B$ and write $B = b_1q_1+\cdots +b_rq_r$ with $r\ge 1$, $b_i>0$ for all $i$ and $q_i\ne q_j$ for all $i\ne j$. By the definition of base locus, one has $h^0(\Oo _X(B))=1$, $h^0(L(-B)) =h^0(L)$ and $L(-B)$ is base-point free. Similarly as above, let $\mathcal R_B$ the set of all isomorphism classes of line bundles $\Oo _X(a_1q_1+\cdots +a_rq_r)$ with $0\le a_i\le b_i$ for all $i$ and $a_1+\cdots +a_r>0$. So $\mathrm{s}^R(u) =1$ for all $R\in \mathcal R_B$ and all $u\in H^0(L)^{\times}$. \\

\item[(iv)] Fix $R\in \Rr$ such that $h^0({R})=1$, say $|R| =\{D\}$ with $D =m_1p_1+\cdots +m_rp_r$, where $m_i>0$ and $p_i\ne p_j$. Fix $u\in H^0(L)^{\times}$. We have $\mathrm{s}^R(u) =1$  if and only if $D_u\supset D$. Assume also $h^0(L\otimes R^{-1}) =1$, say $|L\otimes R^{-1}| =\{E\}$ with $E =a_1q_1+\cdots +a_tq_t$, $t\ge 1$, $a_i>0$, and $q_i\ne q_j$. 

In this case, $\mathrm{s}^R(u) =1$ if and only if $D_u =D+E$. Thus all $u\in H^0(L)^{\times}$ with $\mathrm{s}^R(u)=1$ are proportional. So for $v\in H^0(L)^{\times}$ either $\mathrm{s}^R(v)=1$ or $\mathrm{s}^R(v)=+\infty$.\\

\item[(v)] Now assume $h^0(L) =2$ and that $|L|$ has no base points. This is Example \ref{a3}, but with $X$ being a smooth curve. Take $R\in \Rr$ and so $R\ncong \Oo _X$, $R\ncong L$, $h^0({R})>0$ and $h^0(L\otimes R^{-1})>0$. Since $L$ has no base locus, we have $0< h^0(L\otimes R^{-1}) < h^0(L)=2$ and so $h^0(L\otimes R^{-1})=1$. Then take an effective divisor $D\in |L\otimes R^{-1}|$; so $R\cong L(-D)$. Again, since $L$ is base point free, we have $0< h^0(R) = h^0(L(-D)) < h^0(L)=2$
and so $h^0(R)=1$. 

Fix $w\in H^0({R})^{\times}$ and $v\in H^0(L\otimes R^{-1})^{\times}$. Since $h^0({R})=h^0(L\otimes R^{-1}) =1$, all $u\in H^0(L)^{\times}$ with $s^R(u)=1$ are of the form $u= cw\otimes v$ for some constant $c\ne 0$.
Since $cw\otimes v +c'w\otimes v =(c+c')w\otimes v$ for all constants $c, c'$, we get $s^R(u)=+\infty$ for a general $u\in H^0(L)^{\times}$. Since this is true for all $R\in \Rr$, there is no line bundle $R$ such that $\mathrm{s}^R(u)<+\infty$ for all $u\in H^0(L)^{\times}$. 

Therefore, in this case, strength is more interesting than $R$-strength. Line bundles $L$ as in (v) exist for all smooth curves of genus $g>0$. By Riemann-Roch theorem, when $g=1$, $L$ may be taken to be any degree $2$ line bundle, while for $g\ge 2$ we may take as $L$ a general line bundle of degree $g+1$. For $g=2$, then we may choose $L =\omega _X$; for $g>2$, possible choices for $L$ are those line bundles evincing the gonality of $X$. 

\end{enumerate}
\end{example}

\begin{remark}\label{d5}
Let $X$ be an integral projective variety. Let $L$ and $R$ be non-trivial line bundles on $X$ such that $R\ncong L$, $h^0(R)>0$ and $h^0(L\otimes R^{-1})>0$. One can construct the variety $\Delta _R$ of all $u\in \PP(H^0(L))$ of the form $x\otimes y$ with $x\in \PP(H^0(L\otimes R^{-1}))$ and $y\in \PP(H^0(R))$, i.e. those sections with $R$-strength $\mathrm{s}^R(u) = 1$. Note that, since $X$ is integral, the bilinear map corresponding to the multiplication of sections is non-degenerate. Thus $\Delta_R$ is a projective variety of dimension $\min\lbrace h^0(R)+h^0(L\otimes R^{-1})-2, h^0(L)-1\rbrace$. Let $\mathcal R$ be the set of all line bundles $R$ on $X$ such that  $R\ncong \Oo_X$,$R\ncong L$, $h^0(R)>0$ and $h^0(L\otimes R^{-1})>0$. By Theorem \ref{e1}(i), if $\mathcal R =\emptyset$, then $\mathrm{s}_X(u)=+\infty$ for all $u\in H^0(L)$.
Otherwise, $\mathcal R \ne \emptyset$, and the set $\Delta =\cup _{R\in \mathcal R} \Delta _R$ is a union of irreducible projective varieties, whose dimensions are given above. 

A more intrinsic way to see these sets $\Delta_R$ is as follows. Let $a = h^0(R)-1, b = h^0(L\otimes R^{-1})-1$. Thus $\PP(H^0(R))\cong \PP^a$ and $\PP(H^0(L\otimes R^{-1}))\cong \PP^b$. Let $\mathrm{Seg}: \PP(H^0(L\otimes R^{-1}))\times  \PP(H^0(R)) \to \PP^{ab+a+b}$ denote the Segre embedding. One has $\mathrm{Im}(\mu_{L\otimes R^{-1},R})\cong  H^0(L\otimes R^{-1})\otimes H^0(R)/\mathrm{ker}(\mu_{L\otimes R^{-1},R})\subset H^0(L)$, where $\mu_{L\otimes R^{-1},R}$ is the multiplication map.

If $\mathrm{ker}(\mu_{L\otimes R^{-1}, R})\neq 0$, let $K\subset \PP^{ab+a+b}$ be the projectivisation of $\mathrm{ker}(\mu_{L\otimes R^{-1}, R})$ and let $\ell _K: \PP^{ab+a+b}\setminus K \rightarrow \PP(H^0(L))$ denote the linear projection from $K$. Since $\mu_{L\otimes R^{-1}, R}$ is non-degenerate, one has $K\cap \mathrm{Seg}(\PP(H^0(L\otimes R^{-1}))\times  \PP(H^0(R))) =\emptyset$. Thus $\pi_R = \ell_K\circ \mathrm{Seg}: \PP(H^0(L\otimes R^{-1}))\times  \PP(H^0(R))\rightarrow \PP(H^0(L))$ is a well-defined morphism. So $\Delta_R = \pi_R (\PP(H^0(L\otimes R^{-1}))\times  \PP(H^0(R)))$.  

When the multiplication map $\mu_{L\otimes R^{-1},R} : H^0(L\otimes R^{-1})\otimes H^0(R)\rightarrow H^0(L)$ is injective, the closure of all $u\in \PP(H^0(L))$ with $R$-strength $\le k$ coincides with the $k$-th secant variety of the Segre embedding of $\PP^{a}\times \PP^{b}$, where $a = h^0(R)-1, b = h^0(L\otimes R^{-1})-1$. 

For any integer $k>0$, let $\sigma ^0_k(\{\Delta _R\}_{R\in \mathcal R})$ denote the set of all $u\in
\PP(H^0(L))$ such that $u\in \langle u_1, \cdots, u_k\rangle$, with $u_i\in  \Delta _{R_i}$ and $R_i\in \mathcal R$. This is a constructible set and its closure is the {\it join variety} of all the $\Delta_{R_i}$, which is the locus of points that are limits of sections $u\in \PP(H^0(L))$ of strength at most $k$. 
\end{remark}

\begin{remark}\label{d8}
Keep the notation from Remark \ref{d5}. Note that $R\in \mathcal R$ if and only if $L\otimes R^{-1} \in  \mathcal R$ and that $\Delta _R = \Delta_{L\otimes R^{-1}}$.
Suppose $\#(\mathcal R)=1$, say $\mathcal R =\{R\}$. Then $L\otimes R^{-1} \in \mathcal R$, and so $L\cong R^{\otimes 2}$. 

Let $\mu: H^0(R)^{\otimes 2}\to H^0(L)$, denote the multiplication map. Set
$a = h^0(R), b = h^0(L)$ and let $c = \mathrm{rk}(\mu)$, i.e. the rank of the associated matrix in a basis. Note that $\mu$ vanishes on the $\mathrm{GL}(H^0(R))$-module $\wedge^2H^0(R)$, because the tensor product of sections of line bundles corresponds to addition of Cartier divisors (which is commutative). Hence $\mathrm{Im}(\mu) =\mathrm{Im}(\mu^2)$, where $\mu^2: S^2H^0(R)\to H^0(L)$ is the corresponding symmetric multiplication map. The classification of quadratic forms in $a$-many variables shows that each element of $\mathrm{Im}(\mu^2)$ has strength $\le \lfloor (a+1)/2\rfloor$. If $b>c$, then each element of $H^0(L)\setminus\mathrm{Im}(\mu^2)$ has infinite strength. 
\begin{enumerate}
\item[(i)] Assume $c=\binom{a+1}{2}$, i.e. assume that $\mu^2$ is injective. In this case, for any $u\in \mathrm{Im}(\mu^2)\setminus \{0\}$ the strength of $u$ is the strength of the quadratic form $f$ such that $\mu^2(f) =u$. If $f$ has matrix rank $r$, then $\mathrm{s}_X(u) =\lceil r/2\rceil$.

\item[(ii)] By the Bilinear Lemma \cite[Proposition 1.3]{eis}, one has $c = \dim(\mathrm{Im}(\mu)) \ge 2\dim(R)-1=2a-1$. Moreover, when $R$ is globally generated, there is a classification of all pairs $(X,R)$ with $c=2a-1$ \cite[Proposition 1.6]{eis}. 
\end{enumerate}
\end{remark}

\section{Interlude: over the real numbers}\label{reals}
In this section, varieties and line bundles are defined over $\RR$. Let $X$ be an integral projective variety defined over $\RR$ and let $\sigma: \CC \to \CC$ denote the complex conjugation. The complex conjugation induces a real algebraic isomorphism $\sigma : X(\CC)\to X(\CC)$ and we define the real part of $X$ to be 
\[
X(\RR) = \{x\in X(\CC)\mid \sigma (x)=x\}.
\]

\begin{definition}
For a line bundle $L$ on $X$  defined over $\RR$, one has an induced isomorphism $\sigma: H^0(X(\CC),L) \to H^0(X(\CC),L)$ of complex vector spaces. Define: 
\[
H^0(X(\CC),L)^{\sigma} = \{u\in H^0(X(\CC),L) \mid \sigma (u)= u\}
\] 
to be the real vector space of all $u\in H^0(X(\CC),L)$ defined over $\RR$. One has the equality $\dim _\RR H^0(X(\CC),L)^{\sigma} = \dim _\CC H^0(X(\CC),L)$. 
\end{definition}

As in the previous sections, let $L$ and $R$ be non-trivial line bundles on $X$ defined over $\RR$ and with $h^0(R)>0$, $R\ncong L$ (over $\CC$) and $h^0(L\otimes R^{-1})>0$. Consider the multiplication map 
\[
\mu : H^0(X(\CC),L\otimes R^{-1})\otimes H^0(X(\CC),R) \to H^0(X(\CC),L).
\]
Since $L$ and $R$ are real, $\mu$ commutes with $\sigma$, i.e. for each pair of sections $(u,v)\in H^0(X(\CC),L\otimes R^{-1})\times  H^0(X(\CC),R)$, we have
$\mu (\sigma (u)\otimes \sigma (v)) = \sigma (\mu (u\otimes v))$. Thus $\sigma (\mathrm{Im}(\mu)) = \mathrm{Im}(\mu)$ and $\mathrm{Im}(\mu (H^0(X(\CC), L\otimes R^{-1})^\sigma\otimes H^0(X(\CC),R) ^\sigma)\subseteq H^0(X(\CC), L)^\sigma$.

\begin{definition}
If $u\notin \mathrm{Im}(\mu (H^0(X(\CC), L\otimes R^{-1})^\sigma\otimes H^0(X(\CC),R) ^\sigma)$, we say that $u$ has real $R$-strength $s^{R}_{\RR}(u)=+\infty$.

Otherwise, the real $R$-strength $s^{R}_{\RR}(u)$ of $u$ is the minimal integer $r$ such  that
there exist $u_1,\dots ,u_r\in H^0(R)^\sigma$ and $v_1,\dots ,v_r\in H^0(L\otimes R^{-1})^\sigma$ such that $u =\sum _{i=1}^{r} \mu (u_i\otimes v_i)$.
Analogously, given $(X,L)$ over the reals, one defines the {\it real strength} $\mathrm{s}_{\RR}(u)$ and it is clear that $s_\RR(u)\geq s(u)$.
\end{definition}

\begin{example}\label{r1}
Let $X = \PP^1$ with $L = \Oo_{\PP^1(\CC)}(2)$ and $R = \Oo _{\PP^1(\CC)}(1)$. The section $u = x^2+y^2$ has complex $R$-strength $\mathrm{s}(u)=1$, but real $R$-strength two. Note that, more generally, for an arbitrary projective $X=\PP^n$, complex and real $R$-strength of quadratic forms
only depends on their rank and signature. 
\end{example}

\begin{definition}
Let $L$ be a line bundle defined over $\RR$. For any $u\in H^0(X(\CC),L)$ set $\Ree(u) = (u+\sigma (u))/2$ and $\Imm(u) =
(u-\sigma (u)/2i$. Since $\sigma ^2$ is the identity map, we have $\Ree(u)\in H^0(X(\CC),L)^\sigma$, $\Imm (u)\in
H^0(X(\CC),L)^\sigma$, $u =\Ree(u)+ i\Imm(u)$, where $\Ree(u)$, $\Imm(u)$ are the only elements of $H^0(X(\CC),L)^\sigma$ such that
$u =\Ree(u)+i\Imm(u)$.
\end{definition}

\begin{theorem}\label{r2}
Let $L$ be defined over $\RR$ with $h^0(L)>0$. Suppose that all line bundles $R$ such that $h^0(R)>0$ and $h^0(L\otimes
R^{-1})>0$ are defined over $\RR$. Then $s_\RR(u) \le 2s(u)$ for all $u\in H^0(X(\CC),L)^\sigma$.
\end{theorem}

\begin{proof}
If $s(u)=+\infty$, then $s_\RR(u) =+\infty$. Assume $r=s(u)<+\infty$ and write $u =\mu(u_1\otimes u_1)+\cdots
+\mu (u_r\otimes v_r)$ for some line bundles $R_1,\dots , R_r$ and some $u_i\in H^0(X(\CC),R_i)$, $v_i\in H^0(X(\CC),L\otimes
R_i^{-1})$. By assumption, each $R_i$ and each $L\otimes R_i^{-1}$ is defined over $\RR$, hence we may apply the real part
and the imaginary part to each section. Since $\mu$ is $\CC$-linear, one has $\mu ((\Ree(u_h)+i\Imm(u_h))\otimes (\Ree(v_h)+i\Imm(v_h))
= \mu (\Ree(u_h)\otimes \Ree(v_h)) -\mu(\Imm(u_h)\otimes \Imm(v_h)) + i\mu(\Imm(u_h)\otimes \Ree(v_h)) +i\mu (\Ree(u_h)\otimes \Imm(v_h))$.
Since $\Ree(u) = u$, one has $u =\sum _{h=1}^{r} \mu (\Ree(u_h)\otimes \Ree(v_h)) -\mu(\Imm(u_h)\otimes \Imm(v_h))$ and hence $s_\RR(u)
\le 2s(u)$.
\end{proof}

Theorem \ref{r2} is not true when some line bundle is not defined over $\RR$, as shown by the following example: 

\begin{example}\label{r4}
Let $X = \{x^2+y^2+z^2=0\}\subset \PP^2$ as in Example \ref{fieldext}. The line bundle $L = \Oo _{X(\CC)}(1)$ is defined over $\RR$, and let $R = \Oo_{X(\CC)}(p)$ for a given point $p\in X(\CC)$.  Thus each $u\in  H^0(X(\CC),L)^{\times}$ satisfies $\mathrm{s}^R(u)=1$. 

The line bundle $R$ on $X(\CC)$ is not defined over $\RR$, because every $u\in  H^0(X(\CC),R)$ has a unique point of $X(\CC)$ in its zero-locus, whereas $X(\RR) =\emptyset$. Thus each nonzero $u\in H^0(X(\CC),L)^\sigma$ has infinite real $R$-strength. 
\end{example}

\begin{remark}\label{r5}
Let $(X,L)$ be a pair defined over $\RR$. 
Fix a nonzero $u\in H^0(X(\CC),L)^\sigma$ and let $D= \{u=0\}\subset X(\CC)$, i.e. $D\in |L|$ with $D$ defined over $\RR$.

Write $D = \sum _j b_jE_j$ with $b_j$ positive integers, and where each $E_j$ is an integral divisor over $\RR$ ($E_j$ may be geometrically irreducible or it may split over $\CC$ as $E_j =A\cup \sigma (A)$ with $A\subset X(\CC)$ irreducible over $\RR$). We have $s_\RR(u) =1$, if $\sum b_j\ge 2$. 
\end{remark}

\begin{remark}\label{r6}
Theorem \ref{x3} holds over $\RR$, because the Koszul complex is defined over $\RR$ and it is exact at each point of $X(\CC)$. So clearly 
every result relying solely on Theorem \ref{x3} holds over $\RR$. As an example, in Proposition \ref{d14} appearing in the next section, equip $\PP^2(\CC)$ with its usual real structure. Fix an integer $d\ge 2$ and a nonzero section $u\in H^0(\PP^2(\CC),\Oo _{\PP^2(\CC)}(d))^\sigma$ . If $\{u=0\}\cap \PP^2(\RR) \ne \emptyset$ (this always holds when $d$ is odd), then the proof of Proposition \ref{d14} gives $s_{\RR}(u) \le 2$.
\end{remark}

\section{Strength and $R$-strength for some special varieties}\label{specialvarieties}

\begin{proposition}\label{d2}
Let  $X = \PP^{n_1}\times \cdots \times \PP^{n_k}$, $k\geq 2$, and $L = \Oo_X(d_1,\ldots ,d_k)$, where $n_1=1$ and $d_i\geq 1$ for all $i$. Then $\mathrm{s}_X(u) \le 2$ for all $u\in H^0(L)^{\times}$ and equality holds for general $u$.
\end{proposition}

\begin{proof}
The first assertion follows from Theorem \ref{x3} taking $M = \Oo_X(1,0,\dots ,0)$: $M$ is globally generated, $h^0(M)=2$
and the required cohomological vanishings hold. For the second assertion, it is sufficient to notice that the set of all $u\in H^0(L)^{\times}$ with $\mathrm{s}_X(u)=1$ is a finite union of proper subvarieties in $|L|$. 
\end{proof}

One can further generalize Proposition \ref{d2} as follows: 

\begin{proposition}\label{d3}
Let $(X,L)$ be a pair where $X = \PP^1\times Y$, $Y$ is a connected projective variety. Assume $L = \Oo_{\PP^1}(d)\boxtimes L_Y$, where $d\geq 1$ and $L_Y$ is a line bundle on $Y$ such that $h^0(L_Y)>0$ and $h^i(L_Y)= 0$ for $i=1,2$. Then $\mathrm{s}_X(u) \le 2$ for all $u\in H^0(L)^{\times}$. 
\begin{proof}
Let $M = \Oo_{\PP^1}(1)\boxtimes \Oo_Y$ and so $M$ is globally generated. Note that $h^0(M) = 2$ by K\"{u}nneth formula and because $Y$ is connected.
Moreover, $h^1(M^{-1}\otimes L) = h^1(M^{-2}\otimes L) = h^2(M^{-2}\otimes L) = 0$ again by K\"{u}nneth formula and because $h^1(L_Y)= h^2(L_Y) = 0$. Now, apply Theorem \ref{x3}.
\end{proof}
\end{proposition}

\begin{proposition}
Fix integers $d\ge 2$, $n\ge 2$. Let $Y$ be an $(d-2)$-dimensional projective variety and let $A$, $B$ be globally generated line bundles on $Y$ such that $h^0(A\otimes B^{-1})>0$. Let $(X,L)$ be the pair $X = \PP^n\times Y$ and $L = \Oo _{\PP^n}(d)\boxtimes A$. Suppose $M = \Oo _{\PP^n}(1) \boxtimes B$ 
and that $H^k(A\otimes B^{-k}) = H^k(A\otimes B^{-k-1)}) = 0$ for all $1\leq k\leq m := \min\lbrace n+d-1, h^0(M)\rbrace$.  Then $s_X(u) \le m$ for all $u\in H^0(L)^{\times}$.
\begin{proof}
We first check that the vanishing $H^k(M^{-k}\otimes L) = 0$ for $1\leq k\leq m$ holds. K\"{u}nneth formula on the bundle $L\otimes M^{-k}$ 
reads: 
\[
H^k(\Oo_{\PP^n}(d-k)\boxtimes A\otimes B^{-k}) = \bigoplus_{i+j=k} H^i(\Oo_{\PP^n}(d-k))\otimes H^j(A\otimes B^{-k}). 
\]
Note that the right-hand side vanishes whenever $i>0$. Indeed, for $0<i<n$ or $i>n$ is clear. For $i=n$, one has $j = k - n\leq m - n\leq d-1$; therefore
$d-k = d-n-j > -1-n$ is satisfied and so $H^n(\Oo_{\PP^n}(d-k))=0$. For $i = 0$, we have $j=k$ and the vanishing $H^k(A\otimes B^{-k}) = 0$ for $1\leq k\leq m$ holds by assumption. Similarly, one checks that the vanishing $H^k(M^{-k-1}\otimes L) = 0$ holds. Thus the consequence of Theorem \ref{x3} follows. 
\end{proof}
\end{proposition}

Recall that when $X=\PP^n$, $R =\Oo _{\PP^n}(1)$ and $L = \Oo _{\PP^n}(d)$, $d\ge 2$, the $R$-strength of a homogeneous polynomial $u\in H^0(L)^{\times}$ is its slice 
rank. This is related to the maximal dimension of linear spaces that the hypersurface $\lbrace u = 0\rbrace$ contains. In \cite{BO}, Bik and Oneto exploit
this perspective and we refer to their upcoming work for further details and results in that direction. Here we only give two results to show how linear spaces come into the picture. 

\begin{proposition}\label{d14}
Let $X = \PP^2, R =\Oo _{\PP^n}(1)$, and $d\geq 2$. Then: 
\begin{enumerate}
\item[(i)] $\mathrm{s}^R(u) \leq 2$ for all $u\in H^0(\Oo _{\PP^2}(d))^{\times}$;
\item[(ii)] $\mathrm{s}^R(u) = 2$ for a general $u\in H^0(\Oo _{\PP^2}(d))^{\times}$.
\end{enumerate}
\end{proposition}

\begin{proof}
An irreducible $u\in H^0(\Oo _{\PP^2}(d))^{\times}$ has strength $s(u)>1$. Thus it is sufficient to verify (i). 

Fix $u\in H^0(\Oo _{\PP^2}(d))^{\times}$ and take $p\in \{u=0\}$. Choose a system of homogeneous coordinates $x_0$, $x_1$,
$x_2$ of $\PP^2$ such that $p = (0:0:1)$. To show the statement, it is enough to find degree $d-1$ homogeneous polynomials $u_0$ and $u_1$ such that $u = x_0u_0+x_1u_1$. 

Consider the Koszul complex of the regular sequence $\lbrace x_0,x_1\rbrace$ and twist it by $\Oo _{\PP^2}(d)$. One 
obtains an exact sequence: 
\[
0 \to  \Oo _{\PP^2}(d-2) \longrightarrow \Oo _{\PP^2}(d-1)^{\oplus 2} \longrightarrow \Ii_p(d)\to 0.
\]
Using $h^1(\Oo _{\PP^2}(d-2))=0$, we find that the induced map $H^0(\Oo_{\PP^2}(d-1)^{\oplus 2})\rightarrow H^0(\Ii_p(d))$ is surjective;
this yields that $u = x_0u_0 + x_1u_1$ for some $u_0, u_1\in H^0(\Oo _{\PP^2}(d-1))$. 
\end{proof}

\begin{theorem}\label{d10}
Let $R = \Oo_{\PP^n}(1)$ and $L = \Oo_{\PP^n}(d)$. 
\begin{enumerate}
\item[(i)] Let $m\geq 1$ such that $(n-m)(m+1) < \binom{m+d}{m}$. Then $\mathrm{s}^R(u)\ge n-m+1$ for a general $u\in H^0(L)$ of degree $d\ge 4$; 

\item[(ii)] Let $e$ be the maximal dimension of a linear subspace in $\PP^n$ contained in a general degree $d$
hypersurface $u\in H^0(L)$. Then $\mathrm{s}^R(u)\le n-e$ for all $u\in H^0(L)^{\times}$ and equality holds for a general $u$. 

\end{enumerate}
\end{theorem}

\begin{proof}
(i). For any integer $1\le m <n$, let $\mathbb G(m,n)$ denote the Grassmannian of all $m$-dimensional linear projective subspaces of
$\PP^n$. For any linear space $W\in \mathbb G(m,n)$, one has $h^0(\Oo _W(d)) =\binom{m+d}{m}$ and the restriction map $H^0(L) \to H^0(\Oo_W(d))$ is surjective. Thus $h^0(\Ii _W(d)) = \binom{n+d}{n} -\binom{m+d}{m}$. Define:
\[
\Gamma _m = \{(W,D)\in \mathbb G(m,n)\times |L|\mbox{ such that } W\subset D\}.
\]
Hence $\Gamma _m$ is an irreducible projective variety of dimension $(n-m)(m+1) + \binom{n+d}{n}-\binom{m+d}{m}$. If $(n-m)(m+1) < \binom{m+d}{m}$, i.e., the map $\pi_m: \Gamma_m \to |L|$ given by (the
restriction of) the projection onto the second factor is not dominant, i.e. a general $D\in |L|$ contains no $m$-dimensional linear subspaces of $\PP^n$. Write $D =\{u=0\}$ and $u = \ell_1b_1+\cdots +\ell_rb_r$, where the $\ell_i$ are linear forms. Since $D$ contains no $m$-dimensional linear subspace, we have $\mathrm{s}^R(u) = r \ge n-m+1$. 

(ii). The proof of statement (i) gives $\mathrm{s}^R(u)\ge n+1-(e+1) = n-e$ for a general $u$. Thus it is sufficient to prove the upper bound. Since the Grassmannian $\mathbb G(e,n)$ is complete, every $D\in |\Oo_{\PP^n}(d)|$ contains an $e$-dimensional linear subspace $W$. The $n-e$ linear equations defining $W$ tensored with $L$ give the equation of every such $D \in |\Ii_W(d)|$ written as an $R$-strength decomposition. 
\end{proof}

\begin{remark}\label{d21}
Along the same lines, one shows the following. Let $D\in |\Oo _{\PP^n}(d)|$ and let $u$ be a corresponding section. Let $c$ be the dimension of a complete intersection $Z\subset D$ of $n-c$ hypersurfaces of degrees $<d$. Then $\mathrm{s}(u)\le n-c$.
\end{remark}

\begin{definition}
Let $(X,L)$ be a pair. For a given $u\in H^0(L)^{\times}$ of finite strength (if exists), let $\mathcal S(u)$ denote the set of strength decompositions \eqref{strengthdec} of $u$, up to permuting the summands. 
\end{definition}

\begin{remark}\label{cc2}
Suppose $d\ge 4$. Consider a section $u\in  H^0(\Oo _{\PP^2}(d))^{\times}$ of the form 
\[
u = u_if_{d-i}+v_jg_{d-j}, \ \  i,j\leq \lfloor d/2\rfloor, 
\]
where $u_i$ and $v_j$ are homogeneous polynomials of degrees $i$ and $j$ respectively, with the property that 
their vanishing locus $Z = \{u_i=v_j=0\}$ is a zero-dimensional scheme. 

Since $\dim Z=0$, the sections $u_i$ and $v_j$ form a regular sequence. Thus the Koszul complex gives the exact sequence: 
\begin{equation}\label{eqcc1}
0\to \Oo _{\PP^2}(d-i-j) \to \Oo _{\PP^2}(d-i)\oplus \Oo _{\PP^2}(d-j)\to \Ii _Z(d)\to 0. 
\end{equation}

Note that, since $d\ge i+j-2$, one has $h^2(\Oo _{\PP^2}(d-i-j))= 0$. Moreover, $h^1(\Oo _{\PP^2}(t))=0$ for all $t\in \ZZ$. 

Conversely, fix any complete intersection $Z$ defined by the vanishing of two homogeneous polynomials of degrees $i$ and $j$. By the vanishings above, the resolution (\ref{eqcc1}) gives $h^1(\Ii_Z(d))=0$, i.e. $\dim |\Ii _Z(d)|=\dim |\Oo _{\PP^2}(d)|-ij$. Therefore, for a complete intersection $Z$ defined in degrees $i$ and $j$, the algebraic set of all $D\in |\Oo _{\PP^2}(d)|$ such that $Z\subset D$ has dimension $\dim |\Oo _{\PP^2}(d)|-ij$.

Let $\mathcal Z_{i,j}$ be the set of all  zero-dimensional schemes that are the complete intersection defined in degrees $i$ and $j$. 
The algebraic set $\mathcal Z_{i,j}$ is irreducible and we determine its dimension. By symmetry, we may assume $j\geq i$. 

First, consider the case $j=i$. The Grassmannian of lines in the projective space $|\Oo _{\PP^2}(i)|$ has
dimension $2(\binom{i+2}{2}-2)$. Thus 
\[
\dim Z_{i,i}=(i+2)(i+1)-4 =i^2+3i-2. 
\]
Now, suppose $j>i$. For any $D\in |\Oo _{\PP^2}(i)|$, one has $h^0(\PP^2,\Ii _D(j)) = h^0(\PP^2,\Oo
_{\PP^2}(j-i)) =\binom{j-i+2}{2}$. Thus, looking at the fiber over a general point $D\in |\Oo _{\PP^2}(i)|$,
one derives: 
\[
\dim \mathcal Z_{i,j} =\binom{i+2}{2}+\binom{j+2}{2} -\binom{j-i+2}{2}-1.
\] 
(Note that $\dim \mathcal Z_{i,i+1} >i(i+1)$. Inducting on $j$, we see that $\dim \mathcal Z_{i,j}>ij$ for all $j$.)

Consider the incidence set
\[
\Gamma _{i,j} = \{(Z,D)\in
\mathcal Z_{i,j}\times |\Oo _{\PP^2}(d)|\mid Z\subset D\}.
\]

Let $\pi _1: \Gamma _{i,j}\to \mathcal Z_{i,j}$ and $\pi _2: \Gamma _{i,j}\to |\Oo _{\PP^2}(d)|$ be the two projections onto the factors. 
We have shown that each fiber of $\pi _1$ has dimension $\binom{d+2}{2}-ij-1$. Thus $\Gamma _{i,j}$ is irreducible of dimension
$\binom{d+2}{2}-ij-1 + \dim \mathcal Z_{i,j} > \binom{d+2}{2}-1$. 
\end{remark}

\begin{theorem}\label{setdecomp}
Let $u\in H^0(\Oo_{\PP^2}(d))^{\times}$ be such that $D = \{u=0\}$ is integral. Then the section $u$ may be written as $u = u_if_{d-i}
+ v_jg_{f-j}$, where $\lbrace u_i, v_j\rbrace$ is a regular sequence, if and only if $D\in \pi_2(\Gamma_{i,j})= F(i,j;2)$. If $D\in F(i,j;2)$, then $\Ss (u)$ has dimension $\dim \mathcal Z_{i,j}-ij$.  
\begin{proof}
The first statement is clear from definitions. The dimension of the set $\Ss(u)$ of all decompositions of $u$ follows from the computations in Remark \ref{cc2}. 
\end{proof}
\end{theorem}

\section{Strength on projective spaces and complete intersections}\label{strengthcompleteintersections}

In this section, we consider the pair $(X,L)$ where $X = \PP^n$ and $L = \Oo_{\PP^n}(d)$ with $d\geq 4$. A nonzero homogeneous polynomial $f\in H^0(L)$ such that 
\[
f=\sum _{h=1}^{k} f_hg_h, \mbox{ where } deg (f_h) =i_h\leq \lfloor d/2\rfloor, \mbox{ for some } g_h\in H^0(\Oo_{\PP^n}(d-i_h)),
\] 
is said to possess a {\it strength decomposition} of type $(i_1,\dots ,i_k)\in \NN^{k}$. 

\begin{definition}\label{defE}
Let $E(i_1,\dots ,i_k; n)$ be the set of all $Y\in |L|$ whose corresponding section $f\in H^0(L)$ has a strength decomposition of type $(i_1,\dots ,i_k)$ (not necessarily minimal); let $\overline{E(i_1,\dots, i_k;n)}$ denotes its Zariski closure in $|L|$.  Each $\lbrace f = 0\rbrace = Y\in |L|$ with $\mathrm{s}(f)=k$ is contained in at least one $E(i_1,\dots, i_k; n)$. 

The constructible set of all $\lbrace f = 0 \rbrace = Y\in |L|$ with $\mathrm{s}(f)=k$ is denoted $\mathcal S_k(d,n)$. Since the possible $k$-tuples $(i_1,\ldots, i_k)$ such that $1\leq i_h\leq \lfloor d/2\rfloor$ for $h\in \lbrace 1,\ldots, k\rbrace$ are $\binom{\lfloor d/2\rfloor +k-1}{k}$, the number of irreducible components of $\mathcal S_k(d,n)$ is at most that quantity. 
\end{definition}

\begin{definition}\label{defF}
Let $F(i_1,\dots ,i_k;n)$ be the set of all $Y\in |L|$ whose corresponding section $f\in H^0(L)$ admits a decomposition $f=\sum _{h=1}^{k} f_hg_h$, where $\deg (f_h)=i_h$ for all $h$, and $\lbrace f_1,\dots ,f_k\rbrace$ is a regular sequence. (Equivalently, the scheme $Z = \{f_1=\cdots =f_k=0\}$ has dimension $n-k$ and called a complete intersection.) Let $\overline{F(i_1,\dots, i_k; n)}$ be the Zariski closure of $F(i_1,\dots ,i_k; n)$ in $|L|$. 
\end{definition}

\begin{remark}\label{ooo1}
Fix positive integers $i_1\le \cdots \le i_k\leq \lfloor d/2\rfloor$. Let $\mathcal{Z}_{i_1,\ldots,i_k}$ denote the set of all complete intersections $C\subset \PP^n$ of $k$ hypersurfaces of degrees $i_1,\dots ,i_k$. The set $\mathcal{Z}_{i_1,\ldots,i_k}$ is an irreducible variety.  For any $C\in \mathcal{Z}_{i_1,\ldots,i_k}$, the Koszul complex associated to the $k$ generators of the homogeneous ideal of $C$ gives $h^i(\Ii _C(d)) =0$ for all $i>0$. Thus $h^0(\Ii _C(d)) =\binom{n+d}{n} -h^0(\Oo _C(d))$ and the integer $h^0(\Oo _C(d))$ is the evaluation at the integer $d$ of the Hilbert polynomial $p_C(t)$ of $C$. Since $\mathcal{Z}_{i_1,\ldots,i_k}$ is irreducible, $p_A(t) =p_B(t)$ for all $A, B\in \mathcal{Z}_{i_1,\ldots,i_k}$.  

Hence the integer $\dim |\Ii _A(d)|$ is the same for all $A\in \mathcal{Z}_{i_1,\ldots,i_k}$. Recall $L = \Oo _{\PP^n}(d)$ and define 
\[
\Gamma_{i_1,\ldots,i_k} = \{(Z,D)\in  \mathcal{Z}_{i_1,\ldots,i_k}\times |L|\mbox{ such that } Z\subset D\}.
\]
Since $\mathcal{Z}_{i_1,\ldots,i_k}$ is irreducible and $h^0(\Ii _A(d)) =h^0(\Ii _B(d))$ for all $A, B\in \mathcal{Z}_{i_1,\ldots,i_k}$, $\Gamma_{i_1,\ldots,i_k} $ is irreducible, because it is a projective bundle on an irreducible base. Thus the image $F(i_1,\dots ,i_k;n)$ of $\Gamma_{i_1,\ldots,i_k} $ in $|L|$ by the projection $\mathcal{Z}_{i_1,\ldots,i_k}\times |L|\rightarrow |L|$ is irreducible. 
\end{remark}

\begin{remark}
Let $\lbrace f_1,\dots ,f_k\rbrace$ be a regular sequence and let $Z = \{f_1=\cdots =f_k=0\}$ be its defining complete intersection. Since the homogeneous ideal of $Z$ is generated by $f_1,\dots ,f_k$, the subset of all $\lbrace f= 0\rbrace \in F(i_1,\dots ,i_k; n)$ of degree $d$ with a decomposition $f  =\sum _{h=1}^{k} f_hg_h$ for some $g_h$ is the linear system $|\Ii _Z(d)|$.
\end{remark}

\begin{remark}\label{a1a1} 
Recall $d\geq 4$ is fixed. Let $\lbrace u=0\rbrace\in F(i_1,\dots ,i_k;n)$, where all the indices satisfy $i_h=1$. That means $u =\sum _{h=1}^{k} f_hg_h$ where $f_i$ are linear forms and where $Z = \{f_1=\cdots =f_k=0\}$ is then an $(n-k)$-dimensional linear subspace. Thus $\{u=0\}\supset Z$. The Koszul complex of the regular sequence $f_1,\dots ,f_k$ shows that the converse holds, i.e. if $v$ is a degree $d$ homogeneous polynomial and $\{v=0\}\supset Z$, then $\lbrace v= 0\rbrace \in F(i_1,\dots ,i_k;n)$. Since the Grassmannian of all $(n-k)$-dimensional linear subspaces of $\PP^n$ is a complete variety, each $F(1,\dots ,1;n)$ is closed in $|L|$.  
\end{remark}

\begin{theorem}\label{do1}
Let $d\geq 4$. Fix integers $k\ge 2$, $n \ge k+3$ and $\binom{n-k-2+d}{n-k} \ge k(n-k+1)$. Set $i_j:= 1$ for $1\le j\le k-1$, and $i_k:= 2$.
Let $(j_1,\dots ,j_s)$ be a multi-index with $s\le k$ and $j_h=1$ for all $h=1,\dots ,s$. Then $F(i_1,\dots ,i_k;n)\nsubseteq
\overline{E(j_1,\dots ,j_s;n)}$.
\end{theorem}

\begin{proof}
By Remark \ref{a1a1}, it is sufficient to prove that $F(i_1,\dots ,i_k;n)\nsubseteq F(j_1,\dots ,j_k;n)$, i.e. some hypersurface $X \in F(i_1,\dots ,i_k;n)$, containing
a codimension $k$ complete intersection quadric, does not contain any $(n-k)$-dimensional linear subspace of $\PP^n$. Let $\mathbb G(n-k,n)$ denote the Grassmannian of all $(n-k)$-dimensional linear subspaces of $\PP^n$. (Recall that the variety $\mathbb G(n-k,n)$ has dimension $k(n-k+1)$.)

Let $A$ be a smooth and irreducible complete intersection defined by $k-1$ linear forms and one quadratic form. Then $A$ is a quadric hypersurface of the $(n-k+1)$-dimensional linear space $\langle A\rangle$. To prove the assertion, it is sufficient to prove that a general $D\in |\Ii _A(d)|$ contains no element of $\mathbb G(n-k,n)$.  Since $A$ is a complete intersection, we have $h^1(\Ii _A(d)) =0$ and hence $h^0(\Ii_A(d)) =\binom{n+d}{n}-h^0(\Oo _A(d)) =\binom{n+d}{n}-\binom{n-k+1+d}{n-k+1} +\binom{n-k-1+d}{n-k+1}$. 
If we showed that $h^0(\Ii _{A\cup B}(d)) \le h^0(\Ii _A(d)) -k(n-k+1) -1$ for every $B\in \mathbb G(n-k,n)$, then the set of all $W\in |\Ii _A(d)|$ containing a prescribed $B\in \mathbb G(n-k,n)$ would have codimension $>k(n-k+1)$ in $|\Ii _A(d)|$. Thus $\dim \cup _{B\cup \mathbb G(n-k,n)} |\Ii _{A\cup B}(d)| < \dim |\Ii _A(d))|$, and hence a general $D\in |\Ii _A(d)|$ would not be in $F(j_1,\dots ,j_k;n)$.

Fix $B\in \mathbb G(n-k,n)$. Fix a general hyperplane $H$ of $\PP^n$ containing $B$. Since $A\nsubseteq H$
and $A$ is irreducible, the residual exact sequence of $H$ gives the following exact sequence:
\begin{equation}\label{eqa1}
0 \to \Ii _A(d-1) \to \Ii _{A\cup B}(d) \to \Ii _{(A\cap H)\cup B,H}(d)\to 0.
\end{equation}
Since $A$ is a complete intersection, we have $h^1(\Ii _A(d-1)) =0$. Thus $h^0(\Ii _A(d-1))=\binom{n+d-1}{n}-h^0(\Oo _A(d-1))
=\binom{n+d-1}{n}-\binom{n-k+d}{n-k+1} +\binom{n-k-2+d}{n-k+1}$. So  
\[
h^0(\Ii _A(d)) -h^0(\Ii _A(d-1))
=\binom{n+d-1}{n-1} -\binom{n-k+d}{n-k}+\binom{n-k-2+d}{n-k},
\]
because one uses the identity $\binom{p}{q}-\binom{p-1}{q}=\binom{p-1}{q-1}$. 

Since $B$ is a complete intersection in $H$, we have
$h^1(H,\Ii _{B,H}(d)) =0$ and hence $h^0(H,\Ii _{B,H}(d)) =\binom{n+d-1}{n-1} -h^0(\Oo _B(d)) =\binom{n+d-1}{n-1}
- \binom{n-k+d}{n-k}$. Since $H$ is a general hyperplane containing $B$, we have $\langle A\rangle \nsubseteq H$ (even
when $B\subset \langle A\rangle$). Since $A$ is a smooth quadric hypersurface of $\langle A\rangle$, $A\cap H$ is an irreducible quadric hypersurface of $H\cap \langle A\rangle$. 

Assume $A\cap H\nsubseteq B$. Since $\Ii _{B,H}(d)$ is globally generated, $h^0(H,\Ii _{(A\cap H)\cup B,H}(d))<h^0(H,\Ii _{B,H}(d))$. By \eqref{eqa1}, the last inequality is
\[
h^0(H,\Ii _{(A\cap H)\cup B,H}(d)) = h^0(\Ii _{A\cup B}(d)) - h^0( \Ii _A(d-1)) < h^0(H,\Ii _{B,H}(d)). 
\]
Therefore 
\[
h^0(\Ii _{A\cup B}(d)) < h^0(H,\Ii _{B,H}(d)) + h^0( \Ii _A(d-1)).
\]
Using the previous formulas, one finds that the latter inequality is equivalent to 
\[
h^0(\Ii _{A\cup B}(d)) < h^0( \Ii _A(d)) - \binom{n-k-2+d}{n-k}.
\]
Since $\binom{n-k-2+d}{n-k} \ge k(n-k+1)$ by assumption, we derive $h^0(\Ii_{A\cup B}(d)) < h^0(\Ii _A(d)) -k(n-k+1)$ in this case.

In the case where $A\cap H\subseteq B$ for a general $H$, since both $A$ and $B$ are irreducible, one sees that this inclusion implies $A\subseteq B$. So we would have equality, which is a contradiction. 

Finally, assume $B\subseteq A\cap H$. Since $A$ is a smooth quadric hypersurface of $\langle A\rangle$, its rank (as quadratic form)
is $\dim\langle A\rangle +1 =n-k+2$. Thus $A\cap H$ is a quadric hypersurface of rank $\geq n-k+1$. Since $A\cap H$ contains the $(n-k)$-dimensional linear space $B$, the quadric $A$ is a reducible quadric hypersurface in $\langle A\rangle$. Hence $A$ has rank $\le 2$. Thus $n\leq k+2$, contradicting our assumption.
\end{proof}

Before we proceed, recall that, for any scheme $W\subset \PP^3$, the residual scheme $\mathrm{Res}_H(W)$ of $W$ with respect to $H$
is the closed subscheme of $\PP^3$ with $\Ii _W:\Ii _H$ as its ideal sheaf. For every $k$, we have a residual exact sequence of $H$ in $\PP^3$, 
which we will use repeatedly in the subsequent proofs:
\begin{equation}\label{eqres1}
0 \longrightarrow \Ii _{\mathrm{Res}_H(W)}(k-1)\longrightarrow \Ii _W(k)\longrightarrow \Ii _{W\cap H,H}(k)\longrightarrow 0.
\end{equation}

The next result shows that $(\overline{F(2,2;3)}\setminus F(2,2;3))\cap F(1,4;3) \ne \emptyset$. In particular, this says
that $F(i_1,\ldots, i_k; n)$ and $E(i_1,\ldots, i_k; n)$ might be not closed.

\begin{theorem}\label{cor1}
Assume $d\ge 6$. Then there exists a degree $d$ surface in 
\[
\left(\overline{F(2,2;3)}\setminus F(2,2;3)\right)\cap F(1,4;3).
\]
\end{theorem}

\begin{proof}
We first show two claims. \\

\noindent {\it Claim 1}: Let $H\subset \PP^3$ be a plane. There is a flat family of smooth degree $4$ linearly normal elliptic
curves whose flat limit is a curve $Z\subset \PP^3$ with $Z_{\red} =T\subset H$ being an integral curve of degree $\deg (T)=4$
and two ordinary nodes. Moreover, the nilradical of the sheaf $\Oo _Z$ defines a zero-dimensional scheme of degree $2$ supported at the two
singular points of $T$. \\

\noindent{\it Proof of Claim 1}: This proof is an adaptation of \cite[Example III.9.8.4]{h}, which gives a flat limit $Z'$ of a family of rational normal curves in $\PP^3$, where $Z'_{\red}$ is a nodal plane cubic, with the nilradical of $\Oo _{Z'}$ supported at the singular point of $Z'_{\red}$. 

We regard the $\PP^3$ we are working with as a hyperplane $P$ in a $\PP^4$; we fix another hyperplane $P'$ of $\PP^4$ such that $P\cap P' =H$. Fix a smooth linearly normal elliptic curve $X\subset P$ of degree $4$. For any point $o\in \PP^4\setminus P'$, let $\pi _o: \PP^4\setminus \{o\} \to P'$ denote the linear projection from $o$, i.e. for each $q\in \PP^4\setminus \{o\}$ we define $\pi _o(q) = \langle \{o,q\}\rangle \cap P'$ (note that the line $\langle \{o,q\}\rangle$ contains a unique point of $P'$, because $P'$ is a hyperplane in $\PP^4$ and $o\notin P'$, so the morphism is well-defined). 

For any $o\in \PP^4\setminus P'$, let $\pi^P_o$ denote the projection $\pi_o$ restricted to $P$, i.e. $\pi^P_o: P\setminus \{o\} \to P'$. Note that if $o\notin (P\cup P')$, then $\pi^P_o(X)$ is a smooth linearly normal elliptic curve in $P'$ and there is a linear projective isomorphism of $P$ onto $P'$ sending $X$ onto $\pi _o(X)$. 

Fix a general line $L$ such that $\lbrace a\rbrace = L\cap P$ and $\lbrace a'\rbrace = L'\cap P'$. Each of the points $q\in L\setminus \lbrace a,a' \rbrace$ gives rise to the space curve $\pi^P_q(X)$ as above; since they are all isomorphic, this family parameterized by $\AA^1\setminus\lbrace 0\rbrace$ is flat. Thus it admits a flat limit
\cite[Proposition III.9.8]{h} at $a\in L\cap P$. The projection $\pi^P_a$ is onto the plane $H$ and the image $\pi^P_a(X)\subset H$ is set-theoretically a singular plane curve of degree $4$ with two singular points $p_1,p_2$, because $\pi^P_a$ is birational on $X$ and $X$ has arithmetic genus $1$. Scheme-theoretically, this has two embedded components at the singular points $p_1$ and $p_2$, similarly to \cite[Example III.9.8.4]{h}. \\

\noindent {\it Claim 2}: Let $Q(t) = \lbrace q_{1,t} = q_{2,t} = 0\rbrace$ be a family of smooth degree $4$ linearly normal elliptic
curves, with flat limit $\lim_{t\rightarrow 0} Q(t) = Z$. Then 
\[
\lim_{t\rightarrow 0} H^0(\Ii_{Q(t)}(d)) = H^0(\Ii _Z(d)), \mbox{ for all } d\geq 1.
\]
In other words, the flat limit described above scheme-theoretically is a flat limit of ideals (and therefore the homogeneous ideals of $Q(t)$ ($t\neq 0$) and $Z$ share the same Hilbert function). \\

\noindent{{\it Proof of Claim 2}}: 
By definition of limit, one has the inclusion 
\[
\lim_{t\rightarrow 0} H^0(\Ii_{Q(t)}(d))\subseteq  H^0(\Ii _Z(d)). 
\]
We calculate $h^0(\Ii _Z(d))$. Let $p_1,p_2$ be the two singular points of $T$ and one has $\mathrm{Res}_H(Z)=\{p_1,p_2\}$. Using the exact sequence
\[
0\longrightarrow \Ii _{p_1,p_2}(d-1)\longrightarrow \Oo_{\PP^3}(d-1)\longrightarrow \Oo_{p_1,p_2}(d-1)\longrightarrow 0,
\]
we find $H^0( \Ii _{p_1,p_2}(d-1)) = \binom{d+2}{3}-2$. Now, from the cohomology exact sequence of \eqref{eqres1}, we find:
\[
h^0(\Ii _Z(d)) = h^0( \Ii _{p_1,p_2}(d-1)) + h^0(\Oo_{H}(d-4)) = \binom{d+2}{3}-2+\binom{d-2}{2} = \binom{d+3}{3} - 4d. 
\]
This number coincides with $h^0(\Ii_{Q(t)}(d))$ for any $t\neq 0$ and the equality is proven. \\

By {\it Claim 2}, any section $u\in H^0(\Ii _Z(d))$ is a limit of a sequence of sections in $H^0(\Ii _{Q(t)}(d))$, where each $Q(t)$ ($t\neq 0$) is a complete intersection
of two quadrics. So we have shown that $\lbrace u = 0\rbrace \in \overline{F(2,2;3)}$. Since $T\subset Z$, we have $\Ii _Z\subset \Ii _T$. Hence $\lbrace u= 0\rbrace\in F(1,4;3)$, because $T$ is a complete intersection between a plane and a quartic surface in $\PP^3$. 

We show next that a general $u\in H^0(\Ii _Z(d))$ is such that $\lbrace u = 0\rbrace \notin F(2,2;3)$, i.e. there are no homogeneous polynomials $u_1,u_2, v_1,v_2$ with $\deg (u_1)=\deg (u_2) = 2$, $\deg (v_1)=\deg(v_2) =d-2$, and where $\lbrace u_1,u_2\rbrace$ is a regular sequence (the scheme $B = \{u_1=u_2=0\}$ is a complete intersection curve in $\PP^3$ of degree $4$ and arithmetic genus $1$).

The set $\Bb$ of all such $B$ has dimension $16$, because the Hilbert scheme of linearly normal elliptic curves in $\PP^3$ has dimension $16$ \cite[Corollary 1.45]{hm}. To conclude, it would be sufficient to prove that 
\begin{equation}\label{countparameters}
h^0(\Ii _{Z\cup B}(d)) < h^0(\Ii _Z(d))-16,
\end{equation} 
for all $B\in \Bb$. As we saw, $h^0(\Ii _Z(d)) =\binom{d+3}{3} -4d$. 

Fix $B\in \Bb$ and consider the residual sequence of $H$:
 \begin{equation}\label{eqec}
0 \to \Ii _{\Res_H(Z\cup B)}(d-1) \to \Ii _{Z\cup B}(d)\to \Ii _{(B\cup Z)\cap H,H}(d)\to 0. 
\end{equation}
Recall that $Z\cap H = Z_{\red} = T$. \\

\quad (a) Assume that no irreducible component of $B$ is contained in $H$. Since $B$ is a complete intersection, it has no embedded components. Thus $\Res _H(B) =B$. So $h^0(\Ii _{\Res _H(Z\cup B)}(d-1)) \le h^0(\Ii _B(d-1)) =
\binom{d+2}{3} -4(d-1)$; the last equality holds because  $h^1(\Ii _B(d)) =0$. Since $(B\cup T)\cap H\supseteq T$, we have $h^0(H,\Ii _{(B\cup Z)\cap H,H}(d)) \le h^0(H,\Ii _{T,H}(d)) =h^0(\Oo _H(d-4)) =\binom{d-2}{2}$. Thus $h^0(\Ii _{Z\cup B}(d)) \le \binom{d+2}{3} -4(d-1) +\binom{d-2}{2}$. Inequality \eqref{countparameters} is satisfied 
if $\binom{d+2}{3} -4(d-1) +\binom{d-2}{2}< \binom{d+3}{3} -4d - 16$, which holds if and only if 
\[
4d - 22 >0. 
\]

\quad (b) Assume that $B$ has a degree $1$ component contained in $H$. Thus $\Res _H(B)$ is a degree
$3$ curve with arithmetic genus $\le 1$. For space cubics $Y$ (even reducible and non-connected ones) a case-by-case check shows that for all $k\geq 3$
one has $h^1(\Ii _{Y}(k)) = 0$. Hence $h^1(\Ii _{\Res_H(B)}(d-1)) = 0$ and so $h^0(\Ii _{\Res_H(B)}(d-1)) \leq \binom{d+2}{3} -3(d-1)$ (whether equality holds depends on the arithmetic genus of $\Res _H(B)$). Since $(Z\cup B)\cap H$ contains a plane curve of degree $5$, we have $h^0(H,\Ii _{(Z\cup B)\cap H,H}(d)) \le \binom{d-3}{2}$. Thus \eqref{eqec} gives $h^0(\Ii _{Z\cup B}(d)) \le \binom{d+2}{3} -3(d-1) + \binom{d-3}{2}$. Note that this case depends in fact on strictly less than $16$ parameters, so at most $15$. 
So we may modify inequality \eqref{countparameters} to: 
\[
h^0(\Ii _{Z\cup B}(d)) < h^0(\Ii _Z(d))-15.
\]
This is satisfied if  $\binom{d+2}{3} -3(d-1) + \binom{d-3}{2} < \binom{d+3}{3} -4d - 15$, which holds if and only if 
\[
4d - 23 >0.
\]

\quad ({c}) Assume $B$ has a degree $2$ component contained in $H$. Thus $\deg (\Res _H(B)) =2$ and we have
$h^0(\Ii _{\Res_H(B)}(d-1))= \binom{d+2}{3}-2d+1$. One does a similar calculation as above.  \\

\quad ({d}) Since $B$ is a complete intersection of two quadrics, $H$ cannot contain a subcurve of $B$ of degree $3$ by B\'{e}zout. We have analyzed all possible cases concluding the proof. 
\end{proof} 

\begin{theorem}\label{cor2}
Assume $d\ge 6$. Then there exists a degree $d$ surface in 
\[
(\overline{F(2,2;3)}\setminus F(2,2;3))\cap F(1,3;3) \cap F(1,1;3).
\] 
\end{theorem}

\begin{proof}
We first show two claims. \\

\noindent {\it Claim 1}: Let $H\subset \PP^3$ be a plane. Let $C\subset H$ be a smooth degree $3$ plane curve. Fix $p\in C$ and let
$L\subset \PP^3$ be a line such that $L\cap H =\{p\}$. Let $Z = C\cup L$. Then $Z$ is a flat limit of a family of smooth space
curves with degree $4$ and arithmetic genus $1$, which are complete intersections of two quadric surfaces.\\

\noindent {\it Proof of Claim 1}: The scheme $Z$ is a connected and nodal space curve of degree $4$ and arithmetic genus $1$. Every connected
smooth curve $X\subset\PP^3$ with degree $4$ and genus $1$ is the complete intersection of two quadric surfaces. 
The curve $Z$ is a nodal curve with two smooth irreducible components, $C$ and $L$. The smoothability of such curves inside $\PP^3$ is the content of \cite{hh2}, which we briefly describe. 

Since $Z$ is connected, $\deg (Z)=4$ and $p_a(Z)=1$, $Z$ is smoothable inside $\PP^3$ if and only if it is a flat limit of linearly normal elliptic curves, i.e. of smooth complete intersections of two quadrics. We use that $\{p\} =C\cap L$ is a single point.

Call $N_Z$, $N_C$ and $N_L$ the normal bundles of $Z$, $C$ and $L$, as subcurves of $\PP^3$. We claim that $h^1(N_Z)=0$ and $Z$ is smoothable inside $\PP^3$. 

To see this, we may apply \cite[Theorem 4.1]{hh2} to $Z =C\cup L$ using $C$ (resp. $L$) as the curve $C$ (resp. $D$) in the statement of \cite[Theorem 4.1]{hh2}.  The set $\Delta = C\cap L = \lbrace p\rbrace$. Since $Z$ is a locally complete intersection, $N_Z$ is a rank $2$ vector bundle on $Z$.  Since $C$ (resp. $L$) is a complete intersection of a cubic and a plane (resp. two planes) we have $N_C\cong \Oo _C(3)\oplus \Oo _C(1)$ (resp. $N_L\cong \Oo _L(1)\oplus \Oo _L(1))$. Note that $C$ is an elliptic curve, so $h^1(N_C)=0$. Since $N_L$ is a direct sum of two line bundles of degree one on $L\cong \PP^1$, each vector bundle $E$ on $L$ obtained from $N_L$ by a negative elementary transformation at $p$ is isomorphic to $\Oo _L(1)\oplus \Oo _L$. Thus $h^1(E)=0$. This is the second (and last) condition to be satisfied in the hyphoteses of  \cite[Theorem 4.1]{hh2}. (Let $E$ be a rank $r$ vector bundle of degree $d$ on a curve $X$ and fix a smooth point $p\in X$. A {\it negative elementary transformation} of $E$ at $p$ is a rank $r$ vector bundle $E'$ of degree $d-1$ on $X$, equipped with an inclusion of sheaves $j: E'\hookrightarrow E$ such that $E/j(E')$ is the skyscraper sheaf $\CC_p$ at $p$.) \\

\noindent {\it Claim 2}: Let $Q(t) = \lbrace q_{1,t} = q_{2,t} = 0\rbrace$ be a family of smooth degree $4$ linearly normal elliptic
curves, with flat limit $\lim_{t\rightarrow 0} Q(t) = Z$. Then 
\[
\lim_{t\rightarrow 0} H^0(\Ii_{Q(t)}(d)) = H^0(\Ii _Z(d)), \mbox{ for all } d\geq 1.
\]
In other words, the flat limit described above is a flat limit of ideals (and therefore the homogeneous ideals of $Q(t)$ ($t\neq 0$) and $Z$ share the same Hilbert function). \\

The proof of {\it Claim 2} is very similar to the proof of {\it Claim 2} in Theorem \ref{cor1}, and we omit it. 
Thus we obtain an element $u\in \overline{F(2,2;3)}$. Moreover, since $\Ii _L, \Ii _C\supset \Ii _Z$, then $u\in F(1,1;3)\cap F(1,3;3)$,
as $L$ and $C$ are both complete intersections. 

We now prove that a general $u\in H^0(\Ii _Z(d))$ does not belong to $F(2,2;3)$. Let $\Bb$ be the set of all complete intersection curves $B\subset \PP^3$ of two quadric surfaces. We have $\dim \Bb =16$, as in the proof of Theorem \ref{cor1}. Fix $B\in \Bb$. We use the residual exact sequence \eqref{eqec}. Recall $Z\cap H =C$ is a cubic curve and $\mathrm{Res}_H(Z)=L$. \\

\quad (a) Assume that no irreducible component of $B$ is contained in $H$ and that $L\nsubseteq B$. In this case, we have $\mathrm{Res}_H(Z\cup B) =L\cup B$, with $L\nsubseteq B$. Since $B$ is the complete intersection of two quadrics and $L\nsubseteq B$, one has $\deg (L\cap B)\le 2$. Using the residual exact sequence of a general plane containing $L$, we derive $h^1(\Ii _{B\cup L}(k)) =0$ for all $k\ge 3$. 

Indeed, let $P\subset \PP^3$ be a general plane containing $L$. Since $L\nsubseteq B$ and $P$ is general, $B\cap P$ is zero-dimensional. Since $B$ has no embedded components, we have $\mathrm{Res}_P(B)=B$. Thus the residual exact sequence with respect to the plane $P$ is:  
\begin{equation}\label{eqkk1}
0 \longrightarrow \Ii _B(k-1)\longrightarrow \Ii _{B\cup L}(k)\longrightarrow \Ii _{(P\cap B)\cup L,P}(k) \longrightarrow 0. 
\end{equation}
Since $B$ is a complete intersection, we have $h^1(\Ii _B(k-1)) =0$ for all $k\in \ZZ$. On $P$, consider the residual exact sequence with 
respect to $L$ of $(P\cap B)\cup L$: 
\[
0 \longrightarrow \Ii_{\mathrm{Res}_L((P\cap B)\cup L)}(k-1) \longrightarrow \Ii_{(P\cap B)\cup L}(k) \longrightarrow \Ii _{((P\cap B)\cap L)\cup L, L}(k)\longrightarrow 0.
\]
Note that, since $((P\cap B)\cap L)\cup L = L$, the right-most ideal sheaf is zero. Therefore: 
\[
h^i(P, \Ii_{\mathrm{Res}_L((P\cap B)\cup L)}(k-1)) = h^i(P,  \Ii_{(P\cap B)\cup L}(k)) = h^i(\Ii _{(P\cap B)\cup L,P}(k)),
\]
where the last equality is by definition. Recall $\lbrace p\rbrace = C\cap L$. Since $\deg (B\cap P) = 4$ and $p\in (B\cap P)\cap L$, we have $\deg (\mathrm{Res}_L((P\cap B)\cup L))\le 3$. Thus $h^1(P,\Ii _{\mathrm{Res}_L((P\cap B)\cup L)}(k-1)) =0$ for all $k\ge 3$. So $h^1(\Ii _{B\cup L}(k)) = h^1(\Ii _{(P\cap B)\cup L,P}(k)) = 0$, for $k\geq 3$. 

Since $d\ge 6$ and $h^1(\Ii _{B\cup L}(d-1)) = 0$, from the ideal sheaf exact sequence, we have 
$h^0(\Ii _{B\cup L}(d-1)) =\binom{d+2}{3} -4(d-1) -d +\deg (L\cap B) \le \binom{d+2}{3}-5d+6$. Since $(Z\cup B)\cap H\supseteq C$, we have $h^0(H,\Ii _{(Z\cup B)\cap H,H}(d)) \le \binom{d-1}{2}$. Thus the residual exact sequence of $Z\cup B$ with respect to $H$ 
gives $h^0(\Ii _{Z\cup B}(d))\le \binom{d-1}{2} + \binom{d+2}{3} -5d+6$. We have $\binom{d-1}{2} + \binom{d+2}{3} -5d+6 < \binom{d+3}{3} -4d-16$. 
The last inequality is satisfied if and only if $4d-22>0$. \\

\quad (b) Assume no irreducible component of $B$ is contained in $H$ and that $L\subset B$. The set of all such $B\in \mathcal B$ depends on $10$ parameters, because the line $L$ is fixed and so we are counting the dimension of the Grassmannian of lines of the $6$-dimensional projective space $|\Ii _L(2)|$.  Since $(Z\cup B)\cap H\supseteq C$, we have $h^0(H,\Ii _{(Z\cup B)\cap H,H}(d)) \le \binom{d-1}{2}$. Since $\mathrm{Res}_H(Z\cup B)\supseteq B$, we have $h^0(\Ii _{\mathrm{Res}_H(Z\cup B)}(d-1)) \le \binom{d+2}{3} -4(d-1)$. Thus the residual exact sequence of $Z\cup B$ with respect to $H$ gives $h^0(\Ii _{Z\cup B}(d))
\le \binom{d-1}{2} +\binom{d+2}{3} -4(d-1)$. We have $h^0(\Ii _{Z\cup B}(d)) < \binom{d+3}{3} -4d-10$. Indeed, the last inequality is satisfied if and only if $3d-14>0$. \\

\quad ({c}) Assume $B$ has a degree $1$ component contained in $H$ and that $L\not\subset B$. Thus $(Z\cup B)\cap H$ contains a degree $4$ curve
and hence $h^0(H,\Ii _{(Z\cup B)\cap H,H}(d)) \le \binom{d-2}{2}$. The residual scheme $\mathrm{Res}_H(Z\cup B)$ contains the degree $4$ curve $T = \mathrm{Res}_H(B)\cup L$, where the degree $3$ curve $Y = \mathrm{Res}_H(B)$ is not contained in a plane, because the complete intersection curve $B$ cannot contain any plane cubic by B\'{e}zout. For space cubics $Y$ (even reducible and non-connected ones) a case-by-case check shows that for all $k\geq 3$ one has $h^1(\Ii _{Y}(k)) =0$. Using the residual exact sequence of a general hyperplane containing $L$, we see $h^1(\Ii _T(k)) =0$ for $k\ge 4$. Since $\deg (L\cap Y)\le \deg (L\cap B)\le 2$ and $d\geq 6$, we have $h^0(\Ii _{\mathrm{Res}_H(Z\cup B)}(d-1))\leq h^0(\Ii _T(d-1)) = h^0(\Ii _{Y}(d-1)) - h^0(\Oo _L(d-1)) +\deg (L\cap Y)$ and so $h^0(\mathrm{Res}_H(Z\cup B)(d-1)) \le \binom{d+2}{3} -3(d-1)-d+2$. 

Thus the residual exact sequence for  $Z\cup B$ with respect to $H$ gives $h^0(\Ii _{Z\cup B}(d)) \le \binom{d-2}{2} + \binom{d+2}{3} -3(d-1)-d+2$. 
This inequality is satisfied if and only if $4d-23>0$. \\

\quad (d) Assume $B$ has a degree $1$ component contained in $H$ and that $L\subset B$. As in step ({c}), we obtain $h^0(H,\Ii _{Z\cap H,H}(d)) \le \binom{d-2}{2}$
and $h^0(\Ii _{\mathrm{Res}_H(Z\cup B)}(d-1))\le \binom{d+2}{3} -2d +1 -d +2$. Thus $h^0(\Ii _{Z\cup B}(d)) \le \binom{d-2}{2}
+\binom{d+2}{3}-3d$. As in step (b), here all $B\in \mathcal B$ we are considering depend on $10$ parameters. The inequality $\binom{d-2}{2} +\binom{d+2}{3}-3d +3 < \binom{d+3}{3} -4d-10$ is satisfied if and only if $3d-15>0$. \\

\quad (e) Assume $B$ has a degree $2$ component contained in $H$ and that $L\not\subset B$. So $B\cap H$ contains a degree $2$ curve $E$. Fix a general $o\in H$. Since $h^0(\Ii _B(2)) =2$, there is $Q\in |\Ii _{o\cup B}(2)|$. Since $Q\cap H\supsetneq E$ and $\deg (Q)=2$, B\'{e}zout theorem implies $H\subset Q$. Thus $Q =H\cup H'$ for some plane $H'$. The set of all planes $H'\subset \PP^3$ has dimension $3$ and $H$ is fixed. Since $\dim |\Oo _{\PP^3}(2)| =9$, the set of all $B\in \Bb$ with $B\cap H$ containing a degree $2$ curve has dimension $\le 12$.

Since $(Z\cup B)\cap H$ contains a degree $5$ plane curve, we have $h^0(H,\Ii _{(Z\cup B)\cap H,H}(d)) \le \binom{d-3}{2}$. Since $Y=\mathrm{Res}_H(B)$ has degree $2$ and $\mathrm{Res}_H(Z\cup B)$ is $Y\cup L$, as before, we obtain $h^0(\Ii _{\mathrm{Res}_H(Z\cup B)}(d-1)) \le \binom{d+2}{3}-3d+3$.
Thus $h^0(\Ii _{Z\cup B}(d)) \le \binom{d-3}{2} +\binom{d+2}{3}-3d+3$.  The inequality $\binom{d-3}{2} +\binom{d+2}{3}-3d+3< \binom{d+3}{3} -4d-12$ is satisfied if and only if $4d-20>0$.\\

\quad (f) Assume $B$ has a degree $2$ component contained in $H$ and that $L\subset B$. Thus $\mathrm{Res}_H(Z\cup B) =\mathrm{Res}_H(B)$ is a degree $2$
subcurve of $B$. Therefore $h^0(\Ii _{\mathrm{Res}_H(Z\cup B)}(d-1)) = \binom{d+2}{3} -2d+1$. Since $(Z\cup B)\cap H$ contains a degree $5$ plane curve, we get $h^0(\Ii _{Z\cup B}(d)) \le \binom{d-3}{2}+\binom{d+2}{3} -2d+1$. As in step (b), here all $B\in \mathcal B$ we are considering depend on $10$ parameters. The inequality $\binom{d-3}{2}+\binom{d+2}{3} -2d+1< \binom{d+3}{3} -4d-10$ is satisfied if and only if $3d-16>0$. 
\end{proof}

\begin{proposition}\label{u100}
Let $S$ be a smooth and connected projective surface such that its Picard group satisfies $\mathrm{Pic}(S)\cong \ZZ^2$. Assume that $\mathrm{Pic}(S)$ is freely generated by $[A]$ and $[B]$  with $[A]$ and $[B]$ effective, $B$ ample and $A$ an integral curve with $A^2<0$. Assume the existence of an integral curve $E\subset S$ such that $E^2<0$ and $E\ne A$. Then $[E]$ and $[A]$ form a basis of $\mathrm{Pic}(S)\otimes \QQ$ and the effective cone of $\mathrm{Pic}(S)\otimes \QQ$ is given by the set $h[A]+\ell[E]$ with $h, \ell\in \QQ_{\geq 0}$. Let $[E] =w[A]+z[B]$ with $w, z\in \QQ$. Then the effective cone in $\mathrm{Pic}(S)\otimes \QQ$  is the set of all $a[A]+b[B]$ with
$b\ge 0$ and $az-bw\ge 0$, $a, b\in \QQ$. \end{proposition}

\begin{proof}
Since $A$ and $E$ are integral curves on a surface, with $A^2 < 0$ and $E^2<0$, one has that the linear systems of their multiples satisfy $|\Oo _S(kA)| =\{kA\}$, $|\Oo _S(kE)| =\{kE\}$ for all positive integers $k$. Note that $|\Oo _S(-kA)| =|\Oo _S(-kE)| =\emptyset$ for $k>0$. Since $A$ and $E$ are irreducible and $A\ne E$, one has $A\cdot E\ge 0$. Since $A^2<0$, the preceding sentences imply that $rA \nsim sE$ for all non-zero rational
numbers, i.e. no rational multiples of theirs are linearly equivalent. Since $\mathrm{Pic}(S)\otimes \QQ$ is a two-dimensional $\QQ$-vector space, the classes $[A]$ and $[E]$ form a basis of it. 

To show that the effective cone in $\mathrm{Pic}(S)\otimes \QQ$ coincides with the set $h[A]+\ell[E]$ with $h,\ell \in \QQ_{\ge 0}$, it is sufficient to prove that for
any integral curve $Y$ the unique $h, \ell\in \QQ$ such that $[Y] = h[A]+\ell[E]$ are non-negative. Clearing 
denominators, it is sufficient to prove that if $n, m, u$ are integers with $u>0$ and $uY\in |\Oo _S(nA+mE)|$,
then $n\ge 0$ and $m\ge 0$. Since $|\Oo _S(nA)| =\{nA\}$ for all $n>0$, $|\Oo _S(nE)| =\{nE\}$ for all $n>0$ and $rA\nsim sE$ for any $r, s\in \QQ$,
it follows that $h^0(\Oo _S(nA-E)) =0$ for all $n>0$ and $h^0(\Oo _S(nE-A)) =0$ for all $n>0$.  Thus $ |\Oo _S(nA+mE)|=\emptyset$ if either $n<0$ or $m<0$. 

Now, let $Y$ be an effective divisor and write $[Y] = a[A]+ b[B]$ with $a,b\in \QQ$. By assumption, $[E] =w[A]+z[B]$ with $w, z\in \QQ$.
As shown above, the classes $[A]$ and $[E]$ are not proportional over $\QQ$ up to linear equivalence, so $z\ne 0$. Thus 
\[
[B] =-\frac{w}{z}[A] +\frac{1}{z}[E].
\]
Writing $[Y] = h[A] + k[E]$, we obtain $h = a - \frac{bw}{z}$ and $k = b/z$. Since $B$ is effective, by the first part of the proof, we cannot have $z<0$, and so $z>0$. 
Thus $Y$ is effective if and only if $b\ge 0$ and $az-bw\ge 0$. 
\end{proof}

Recall that the constructible set of all $\lbrace f = 0 \rbrace = Y\in |\Oo_{\PP^n}(d)|$ with strength $\mathrm{s}(f)=2$ is denoted $\mathcal S_2(d,n)$. 

\begin{theorem}\label{d31}
For all $n\ge 3$ and $d\ge 4$, the constructible set $\mathcal S_2(d,n)$ is reducible.
\end{theorem}

\begin{proof}

We show that $F(1,2;n) \nsubseteq \overline{F(1,1;n)}=F(1,1;n)$; recall that they are irreducible by Remark \ref{ooo1}. 
We prove that the general point of $F(1,2;n)$ is not in $F(1,1;n)$ and denote this assertion $(\ast)_n$. \\

\begin{quote}
\noindent {\it Claim}. If $(\ast)_3$ is true, then $(\ast)_n$ is true. \\

\noindent {\it Proof of the claim}. We do induction on $n$. Thus we assume $n>3$ and that $(\ast)_{n-1}$ is true.
Fix a hyperplane $H\subset \PP^n$. We will write $F(1,2;H)$ and $F(1,1;H)$ instead of $F(1,2;n-1)$ and $F(1,1;n-1)$, because we will see hypersurfaces of $\PP^{n-1}$ as hypersurfaces of $H$. Fix $o\in \PP^n\setminus H$. For any $Y\in F(1,2;n)$ we fix an integral $T_Y\subset Y$ with $T_Y\subset \PP^n$ complete intersection of a hyperplane and a quadric hypersurface. There is a non-empty Zariski open subset $U$ of $F(1,2;n)$ such that $\dim Y\cap H =n-2$ and $\dim T_Y\cap H =n-3$. Note that for each $Y\in U$, we have $Y\cap H\in F(1,2;H)$. Thus this defines a morphism $\phi: U\to F(1,2;H)$. Fix a general $G\in F(1,2;H)$ and let $C\subset G$ be a complete intersection of a hyperplane in $H$ and a quadric hypersurface in $H$. For a general $G\in F(1,2;n)$, $G$ and $C$ are irreducible. Let $C_o(G)\subset \PP^n$ (resp. $C_o(C)$) be the cone with vertex $o$ over $G$ (resp. $C$). Since $G$ and $C$ are irreducible, $C_o(G)$ and $C_o(C)$ are irreducible. Since $C_o(G)\cap H =G$ and $C_o(C)\cap H = C$, we have $\phi(C_o(G)) = G$ and so $\phi$ is dominant. 

Assume that $(\ast)_n$ fails and take a general $F\in U$. By assumption, $F\supset L$ with $L$ a linear space of dimension $n-2$.
Then $L\cap H$ is a linear space of dimension at least $n-3$. Thus $F\cap H\in F(1,1;H)$. Since $\phi$ is dominant, $(\ast)_{n-1}$ fails, contradicting our assumption. 
\end{quote}

Thus from now on we assume $n=3$. Note that $E(1,1;3)$ is the union of $F(1,1;3)$ and the set of all reducible surfaces containing a plane. 

We show that $F(1,2; 3)\nsubseteq \overline{F(1,1;3)}$, concluding the proof that $\mathcal S_2(d,3)$ and hence $\mathcal S_2(d,n)$ is reducible. Let $D$ be a smooth conic, say $D =\{\ell = q=0\}$ with $\ell$ a linear form and $q$ a quadratic form. The Koszul complex of
the regular sequence shows that $S\in |\Ii _D(d)|$ if and only if $S$ has an equation $f = \ell f_1 +q f_2$ with $\deg (f_1)=d-1$
and $\deg (f_2)=d-2$. By \cite{lop2}, a general $S\in |\Ii _D(d)|$ has Picard group freely generated by
$D$ and $\Oo _S(1)$. We check that $S$ contains no line and hence $S\notin E(1,1;3)$. 

Suppose $S$ contains a line $L$ and write $\Oo _S(L) \cong \Oo _S(b)(aD)$ with $(a,b)\in \ZZ^2$. Since $\deg (D) =2$, $\Oo _S(1)\cdot \Oo _S(1) =d$ and $\deg (L) =1$, we get the integral equation: 
\begin{equation}\label{eqb1}
bd +2a =1. 
\end{equation}
To derive a contradiction using \eqref{eqb1}, it is sufficient to prove that $b \ge 0$ and $b \ge -a$. Since $D$ is a smooth conic, we have $\omega _D \cong \Oo _D(-1)$ and $\deg (\omega _D)=-2$. As $S\subset \PP^3$ is a degree $d$ surface, the adjunction formula gives $\omega _S \cong \Oo _S(d-4)$. Now, $D$ is a curve on the smooth surface $S$, and so adjunction gives $\deg (\Oo _D(d-4)) +D^2 =-2$, i.e. $D^2 = 6-2d$. We have $D^2<0$, because $d\ge 4$. Moreover, a Bertini type argument shows that a general $E\in |\Oo _S(1)(-D)|$ is smooth and integral.  

Let $A = D$ and $B = \Oo _S(1)$ be the ones in the assumptions of Proposition \ref{u100}. 
Since a general $E\in |\Oo _S(1)(-D)|$ is a smooth plane curve of degree $d-2$, the adjunction formula gives $\omega _E \cong \Oo _E(d-5)$ and hence
$\deg (\omega _E) = (d-2)(d-5)$. By the adjunction formula on $S$ we have
$\deg (\omega _E) =\deg (\Oo _E(d-4)) + E^2$ and hence $E^2<0$. By definition, 
$D+E\in |\Oo_S(1)|$, i.e. in the setting of Proposition \ref{u100}, one has $w=-1$ and $z= 1$.  Since $E$ is effective, by the conclusion of Proposition \ref{u100}, we obtain $a\ge 0$ and $b\ge - a$.  \end{proof}

\begin{remark}\label{NLlocus}
The {\it Noether-Lefschetz locus} $\mathcal{NL}_d\subset |\Oo_{\PP^3}(d)|$ is the set of all smooth degree $d$ surfaces $S$ such that  $\mathrm{Pic}(S) \neq \ZZ\Oo _S(1)$. This locus is a countable union of algebraic sets and a classical result of Noether affirms that $\mathcal{NL}_d$ does not coincide with 
the set parametrizing all degree $d$ smooth surfaces $U\subset |\Oo_{\PP^3}(d)|$.  

A smooth $S\in |\Oo_{\PP^3}(d)|$ is in $E(1,1;3)$ if and only if it contains a line. The set of all smooth $S\in |\Oo_{\PP^3}(d)|$ containing a line is a codimension $d-3$ component of $\mathcal{NL}_d\subset |\Oo_{\PP^3}(d)|$ \cite[Theorem 1]{green}; see also \cite{voisin,xu}.  Moreover, the set of all smooth $S\in |\Oo_{\PP^3}(d)|$ containing a complete intersection curve different from a line has higher codimension: for $d\geq 5$, the next largest component of $\mathcal{NL}_d\subset |\Oo_{\PP^3}(d)|$ consists of smooth surfaces containing a conic, which has codimension $2d-7$ \cite[Th\'{e}or\`{e}me 0.2]{voisin}. 

Notice that the set of all smooth $Y\in F(i,i;3)$, for $i=1,\ldots, \lfloor d/2\rfloor$ of degree $d$ is contained in $\mathcal{NL}_d$. Indeed, by definition any such $Y$ 
contains a complete intersection curve $Z\subset Y$ whose canonical bundle is $\omega_Z\cong \Oo_{Z}(2i-4)$. If $Y$ were in the complement of $\mathcal{NL}_d$,
one would have $\mathrm{Pic}(Y)\cong \ZZ\Oo_Y(1)$. Therefore, any curve $C\subset Y$ would be a complete intersection between $Y$ and another surface $Y'$ of some degree $d'$: however, $Z\subset Y$ cannot be realized as an intersection $Y\cap Y'$, as $\omega_{Y\cap Y'}\cong \Oo_{Y\cap Y'}(d+d'-4)$ and $2i<d+d'$. 
\end{remark}

We conclude with the following example proving it is possible to have $\mathrm{s}(u) < \mathrm{s}^R(u)$ for all line bundles $R$ in Definition \ref{mstrength}: 

\begin{example}\label{newd1}
Let $d\geq 5$. There exists $\lbrace u = 0\rbrace = S\in F(1,2;3)$ (smooth with strength $\mathrm{s}(u)= 2$)
which is neither in $F(i;3)$ for any $i$ (because $u$ can be chosen irreducible) nor in $F(i,i;3)$, for any $i=1,\dots ,\lfloor d/2\rfloor$. Indeed, for $i=1$ this is the content Theorem \ref{d31}; for $i>1$, by Remark \ref{NLlocus}, the smooth elements of $F(1,2;3)$ are in the unique submaximal component of the Noether-Lefschetz locus $\mathcal{NL}_d$, whereas the smooth elements of $F(i,i;3)$ sits in higher codimensions. So $u$ has $R$-strength $>2$ for all line bundles $R$ on $\PP^3$.
\end{example}


\begin{thebibliography}{99}
\bibitem{AH} T. Ananyan and M. Hochster, {\it Small subalgebras of polynomial rings and Stillman's conjecture}, J. Amer. Math. Soc. 33 (2020), 291--309.  

\bibitem{BDE} A. Bik, J. Draisma, and R. H. Eggermont, {\it Polynomials and tensors of bounded strength}, Commun. Contemp. Math. 21(7), 2019. 

\bibitem{BO} A. Bik and A. Oneto, {\it On the strength of general polynomials}, preprint, 2020. 

\bibitem{DES} H. Derksen, R. Eggermont, and A. Snowden, {\it Topological noetherianity for cubic polynomials}, Algebra Number Theory, 11(9):2197--2212, 2017.

\bibitem{EG} J. S. Ellenberg and D. Gijswijt, {\it On large subsets of $\mathbb F^n_q$ with no three-term arithmetic progression}, Ann. of Math. 185(1), 339--343, 2017. 
      
\bibitem{ESS} D. Erman, S. V Sam, and A. Snowden, {\it Big polynomial rings and Stillman's conjecture}, Invent. Math. 218, 413--439, 2019. 

\bibitem{eis} D. Eisenbud, {\it Linear sections of determinantal varieties}, Amer. Math. J. 110 (1988), no. 3, 541--575.

\bibitem{green} M. Green, {\it Components of maximal dimension in the Noether-Lefschetz locus},  J. Differential Geom. 29 (1989), no. 2, 295--302.
      
\bibitem{hm} J. Harris and D. Morrison, {\it Moduli of curves}, Springer, Berlin, 1998. 

\bibitem{hh2} R. Hartshorne and A. Hirschowitz, {\it Smoothing algebraic space curves}, Algebraic Geometry, Sitges 1983, 98--131, Lecture Notes in Math.
1124, Springer, Berlin, 1985.

\bibitem{h} R. Hartshorne, {\it Algebraic Geometry}, Springer-Verlag, Berlin--Heidelberg--New York, 1977.

\bibitem{KZ20} D. Kazhdan and T. Ziegler, {\it Properties of high rank subvarieties of affine spaces}, \texttt{arXiv:1902.00767}, 2019. 

\bibitem{KZ} D. Kazhdan and T. Ziegler, {\it On Ranks of Polynomials}, Algebras and Representation Theory, vol. 21, 1017--1021, 2018. 

\bibitem{laz1} R. Lazarsfeld, {\it Positivity in Algebraic Geometry}, Vol. I, Springer, Berlin, 2004.

\bibitem{laz2} R. Lazarsfeld, {\it Positivity in Algebraic Geometry}, Vol. II, Springer, Berlin, 2004.

\bibitem{lop2} A. F. Lopez, {\it Noether-Lefschetz theory and the Picard group of projective surfaces}, Mem. Amer.
Math. Soc. 89 (1991).

\bibitem{SawinTao} W. F. Sawin and T. Tao, {\it Notes on the slice rank of tensors}, Notes available at the website \texttt{https://terrytao.wordpress.com/2016/08/24/notes-on-the-slice-rank-of-tensors/}, 2016. 

\bibitem{Schwede} K. Schwede, {\it Generalized divisors and reflexive sheaves}, Notes available at the website \texttt{www.math.psu.edu/schwede/Notes/GeneralizedDivisors.pdf}, 2010.

\bibitem{voisin} C. Voisin, {\it Composantes de petite codimension du lieu de Noether-Lefschetz}, (French) Comment. Math. Helv. 64 (1989), no. 4, 515--526.

\bibitem{xu}  G. Xu, {\it Subvarieties of general hypersurfaces in projective space}, J. Differential Geom. 39 (1994), no. 1, 139--172.

\end{thebibliography}
\end{document}